\documentclass[10pt,a4paper,oneside,reqno]{amsart}
\usepackage{setspace}
\usepackage[totalwidth=14.5cm,totalheight=22.6cm]{geometry}
\usepackage[T1]{fontenc}
\usepackage[english]{babel}
\usepackage[autostyle]{csquotes}
   \MakeAutoQuote{‘}{’}
   \MakeOuterQuote{"}
\usepackage{graphicx}
\usepackage[all,cmtip]{xy} 
\SelectTips{eu}{12} 
\usepackage{caption}
\usepackage{epsfig}
\usepackage{amsmath,amssymb,amscd,amsthm}
\usepackage{subfigure}
\usepackage{latexsym}
\usepackage{graphics}
\usepackage{colortbl} 
\usepackage{color}
\usepackage{ae}
\usepackage{enumitem}
\usepackage[all]{xy}
\usepackage{scalerel}
\usepackage{tikz}
\usepackage{cancel}
\usepackage{bbm,amsbsy}
\usepackage{comment}
\usepackage{mathrsfs}  
\usepackage[colorlinks=true,pagebackref=true]{hyperref}
\usepackage[normalem]{ulem}
\usepackage{nicefrac}
\usepackage{xfrac}
\usepackage{faktor}
\usepackage{tikz-cd} 
\usepackage{systeme}
\usepackage{nomencl}
\usepackage[normalem]{ulem}
\makenomenclature

\makeatletter
\newtheorem*{rep@theorem}{\rep@title}
\newcommand{\newreptheorem}[2]{%
\newenvironment{rep#1}[1]{%
 \def\rep@title{#2 \ref{##1}}%
 \begin{rep@theorem}}%
 {\end{rep@theorem}}}
\makeatother

\makeatletter
\newtheorem*{rep@cor}{\rep@title}
\newcommand{\newrepcor}[2]{%
\newenvironment{rep#1}[1]{%
 \def\rep@title{#2 \ref{##1}}%
 \begin{rep@cor}}%
 {\end{rep@cor}}}
\makeatother

\makeatletter
\newtheorem*{rep@prop}{\rep@title}
\newcommand{\newrepprop}[2]{%
\newenvironment{rep#1}[1]{%
 \def\rep@title{#2 \ref{##1}}%
 \begin{rep@prop}}%
 {\end{rep@prop}}}
\makeatother

\newtheorem{cor}{Corollary}[section]
\newtheorem{corx}{Corollary}

\newtheorem{theorem}[cor]{Theorem}

\newtheorem{prop}[cor]{Proposition}

 \allowdisplaybreaks
\newrepcor{cor}{Corollary}
\newreptheorem{theorem}{Theorem}
\newrepprop{prop}{Proposition}

\newtheorem{lemma}[cor]{Lemma}

\theoremstyle{definition}
\newtheorem{defi}[cor]{Definition}
\theoremstyle{remark}

\newtheorem{remark}[cor]{Remark}
\newtheorem*{remark*}{Remark}

\newtheorem*{notation*}{Notation}
\theoremstyle{plain}
\newtheorem{principal}{Theorem}

\newcommand{\R}{\mathbb{R}}

\newlist{steps}{enumerate}{1}
\setlist[steps, 1]{itemsep=8pt,leftmargin=0cm,itemindent=.5cm,labelwidth=\itemindent,labelsep=0cm,align=left,label = \textbf{\emph{Step \arabic*}:\,}}

\makeatletter
\newcommand{\myitem}[1]{%
\item[#1]\protected@edef\@currentlabel{#1}%
}
\makeatother

\newcommand{\psl}{\mathrm{P}\mathrm{S}\mathrm{L}(2,\mathbb{R})}
\newcommand{\ads}{\mathbb{A}\mathrm{d}\mathbb{S}^3}

\newcommand{\isom}{\mathrm{Isom}_0}

\newcommand{\hh}{\mathbb{H}^2_{*}}

\newcommand{\fhads}{\mathbb{A}\mathrm{d}\mathbb{S}^{f,h}}
\newcommand{\fhadss}{\mathbb{A}\mathrm{d}\mathbb{S}^{f,h}_*}

\newcommand{\C}{\mathbb{C}}

\begin{document}\raggedbottom

\setcounter{secnumdepth}{3}
\setcounter{tocdepth}{2}

\title[Domination and AdS geometry]{Domination between non-Fuchsian representations and anti-de Sitter geometry}

\author[Farid Diaf, Abderrahim Mesbah, Nathaniel Sagman]{Farid Diaf \and Abderrahim Mesbah \and Nathaniel Sagman}

\address{Farid Diaf: Institut de Recherche Mathématique Avancée, UMR 7501, Université de Strasbourg et CNRS, 7 rue René Descartes, 67000 Strasbourg, France.} 
\email{f.diaf@unistra.fr}

\address{Abderrahim Mesbah:
Beijing Institute of Mathematical Sciences and Applications, Beijing, China.}
\email{abderrahimmesbah@bimsa.cn}

\address{Nathaniel Sagman: Department of Mathematics, University of Toronto, 40 St George Street, Toronto,
ON M5S 2E4, Canada.} 
\email{nathaniel.sagman@utoronto.ca}

\thanks{}

\begin{abstract}
Motivated by work of various authors on domination between surface group representations, harmonic maps, and $3$-dimensional anti-de Sitter geometry, we study a new domination problem between non-Fuchsian representations of closed surface groups. We solve the problem for representations that admit branched harmonic immersions, and we show that, outside of this case, the problem cannot always be solved. We then show that a dominating pair gives rise to an anti-de Sitter $3$-manifold with singularities, and we construct large families of branched anti-de Sitter $3$-manifolds.
\end{abstract}
\maketitle
\tableofcontents
\section{Introduction}
Let $\Sigma$ be an oriented surface with fundamental group $\Gamma$ and let $ (\mathbb{H}^2, \sigma) $ be the hyperbolic space of constant curvature $-1$. As is standard, the group $\psl$ acts on $ (\mathbb{H}^2, \sigma)$ by M{\"o}bius transformations. 
We say that a representation $j:\Gamma\to \psl$ \textit{strictly dominates} another representation $\rho: \Gamma\to \psl$ if there exists a map $F:\mathbb{H}^2\to \mathbb{H}^2$ with equivariance $F\circ j(\gamma)=\rho(\gamma)\circ F$, $\gamma\in \Gamma$, that contracts hyperbolic distance. Formally, the latter condition means that there exists $\lambda<1$ such that for all $p,q\in \mathbb{H}^2,$ $$d_\sigma(F(p),F(q))\leq \lambda d_\sigma(p,q).$$
Domination has attracted considerable attention and has been studied for maps between spaces other than $\mathbb{H}^2$. When $\Sigma$ is closed, it is equivalent to strict domination of the simple translation length spectrum \cite[Theorem 1.8]{GK_lamination} (see also \cite{ALVAREZ201927} for a related dynamical application). The original motivation for domination comes from $3$-dimensional anti-de Sitter geometry. After Thurston's geometrization program, it is natural to study Lorentzian structures on 3-manifolds, and the study of the $3$-dimensional anti-de Sitter space $\ads$ fits into the broader study of Clifford-Klein forms. Following a long line of research (see \cite{Kulkarni_Raymond}, \cite{Salein}, \cite{Kassel}, \cite{GK_lamination}, etc.), it was proved that when $\Sigma$ is closed, a strictly dominating pair is equivalent to a closed $3$-manifold locally modeled on $\ads$ (more on this below).

A representation $ j : \Gamma \to \mathrm{PSL}(2,\mathbb{R}) $ is called \emph{Fuchsian} if it is discrete and faithful; equivalently, it is the holonomy of a hyperbolic structure on $\Sigma $. In \cite{Deroin_thlozan_domination}, the authors used harmonic maps to show that any non-Fuchsian representation of a closed surface group is dominated by a Fuchsian one.  The same domination result was proved in \cite{GKW_folded} using folded hyperbolic surfaces.  The results of \cite{Deroin_thlozan_domination} and \cite{GKW_folded} have been extended in many directions, including for surfaces with punctures (see \cite{Nathaniel_JDG}, \cite{Almost_domination_Nat}, \cite{Gupta_Su}), for representations to Lie groups of higher rank (see, for instance, \cite[Theorem 3]{Tho_Duke}, \cite[Theorem 4]{CTT}, \cite[Theorem 1.8 and Conjecture 1.11]{Dai_Li_cyclic}, \cite{Dai_Li_Domination}, \cite{Barman_Gupta}), and for actions on $\textrm{CAT}(-1)$ metric spaces (see \cite{MartinBaillon2020DominatingCS}).

In this paper, we study notions of domination between \textit{non-Fuchsian} representations of closed surface groups.
We refer to a map $f:\widetilde{\Sigma}\to \mathbb{H}^2$ as just $\rho$-equivariant (or, if $\rho$ is unspecified, equivariant) if for all $\gamma \in \Gamma$, $f\circ \gamma=\rho(\gamma)\circ f$, where the first $\Gamma$-action is via deck transformations. Toward the definition below, we note that, given an equivariant map $f:\widetilde{\Sigma}\to\mathbb{H}^2$, the pullback metric descends from $\widetilde{\Sigma}$ to $\Sigma$. For $C^1$ equivariant maps $f$ and $h$, we write $f^*\sigma< h^*\sigma$ to mean that $f^*\sigma(v,v)\leq h^*\sigma(v,v)$ for every unit tangent vector $v$, and that the inequality is strict when $h^*\sigma$ is positive definite. We also recall that when $\Sigma$ is closed and equipped with a Riemann surface structure $X$, and $\rho$ is \emph{reductive}, there exists a $\rho$-equivariant harmonic map $f:\widetilde{X}\to \mathbb{H}^2$; the map is not always unique, but the pullback metric $f^*\sigma$ is (see Section \ref{sec: hmaps definitions}).
\begin{defi}\label{def: domination pullback}
Let $j, \rho : \Gamma \to \mathrm{PSL}(2,\mathbb{R})$ be representations. 
\begin{enumerate}
    \item We say that $j$ \emph{dominates} $\rho$ in the \emph{pullback sense} if there exist $C^1$ maps $h : \widetilde{\Sigma} \to \mathbb{H}^2$ and $f : \widetilde{\Sigma} \to \mathbb{H}^2,$ which are respectively $j$- and $\rho$-equivariant, such that $ f^*\sigma < h^*\sigma.$
    
    \item When $\Sigma$ is closed and both $j$ and $\rho$ are reductive, we say that $j$ \emph{dominates} $\rho$ in the \emph{harmonic maps sense} if there exists a Riemann surface structure $X$ on $\Sigma$ such that the $\rho$-equivariant and $j$-equivariant harmonic maps $ f, h : \widetilde{X} \to \mathbb{H}^2$ satisfy $ f^*\sigma < h^*\sigma.$
    
\end{enumerate}
\end{defi}

The conditions are motivated by \cite{Deroin_thlozan_domination} (see Section \ref{sec: intro_domination} below). Strict domination implies domination in the pullback sense, and it is equivalent to both domination in the pullback sense and in the harmonic maps sense when $j$ is a Fuchsian representation of a closed surface group. The first condition is a little flimsy--one has a lot of freedom with $f$ and $h$--while the second one is more rigid.

\subsection{Main results} In this paper, the two main questions we address are the following.
\begin{enumerate}
    \item \label{q1} Given $\rho$, can we dominate it in the pullback or harmonic maps sense by a non-Fuchsian representation $j$?
    \item \label{q2} What geometry do dominating pairs encode?
\end{enumerate}
Concerning (\ref{q1}), we provide a complete answer for certain classes of representations (Theorem \ref{principal_C} and Corollary \ref{corc}), and we also find that there is a menagerie of
interesting examples that are worthy of further study (Theorem \ref{principal_D}). Our results toward (\ref{q1}), together with further context relating to harmonic maps, are contained in Section \ref{sec: intro_domination} below. At this stage, we can state one positive result, which extends \cite[Theorem A]{Deroin_thlozan_domination}. Representations from a closed surface group to $\psl$ are organized into connected components according to an integer invariant called the \emph{Euler number} (see Section \ref{sec: energies}). The set of possible Euler numbers is $[-2g+2,2g-2]\cap\mathbb{Z}$ and a representation $\rho$ is Fuchsian precisely when $\textrm{eu}(\rho)=\pm (2g -2)$, see \cite{goldmanthesis}. Note that when $\textrm{eu}(\rho)\neq 0$, the representation $\rho$ is reductive and the $\rho$-equivariant harmonic map is unique. Our first domination result concerns the case where this harmonic map is a \emph{branched immersion}. Recall that a map $f$ between surfaces has a branch point at a point $p$ if one can find local coordinates $z$ and $w$ centered at $p$ and $f(p)$ respectively on which $f$ takes the form $z\mapsto w=z^n$. We say that $f$ is a branched immersion if it is an immersion apart from at a (discrete) set of branch points.

\begin{principal}\label{principal_A}
    Let $X$ be a closed Riemann surface of genus $\geq 2$ with fundamental group $\Gamma$ and let $\rho:\Gamma\to \psl$
     be a representation with $0< \textrm{eu}(\rho)<2g-2$ such that the unique equivariant harmonic map $f:\widetilde{X}\to (\mathbb{H}^2,\sigma)$ is a branched immersion. Then, for every $k$ between $1$ and $2g-2-\textrm{eu}(\rho)$, we can choose a representation $j:\Gamma\to \psl$ with $\textrm{eu}(j)=\textrm{eu}(\rho)+k$ and equivariant harmonic map $h$ such that $f^*\sigma<h^*\sigma$.
\end{principal}
The assumption that the Euler class is positive is not restrictive, since its sign can be flipped by applying an outer automorphism of the fundamental group. Theorem \ref{principal_A} is essentially a consequence of Theorem \ref{principal_C} below. Pairs $(\rho,f)$ as above are constructed and parametrized using Propositions \ref{prop: phi and divisors intro} and \ref{prop: branched immersions} below. Equivariant branched immersions are notable because they are equivalent to hyperbolic cone structures on $\Sigma$ (see \cite{Tan1994} for explanation). The equivariant harmonic map $h$ associated with $j$ is also a branched immersion, so Theorem \ref{principal_A} can be alternatively cast as a domination result between hyperbolic cone structures. The choices of $j$ of Euler number $\textrm{eu}(\rho)+k$ are parametrized by effective divisors of degree $k$ that dominate the branching divisor of $f$. For $k=2g-2-\textrm{eu}(\rho),$ there is just one choice, and this is the Fuchsian representation from \cite[Theorem A]{Deroin_thlozan_domination}.

Theorem \ref{principal_C} in fact produces many examples of domination data $(\rho,j,f,h)$ where $f$ is not a branched immersion, but there is no immediate geometric description for these maps, so we don't include them in Theorem \ref{principal_A}. Nevertheless, these other examples are important for our applications to anti-de Sitter geometry.

For the second question (\ref{q2}), we state our main result here, but we first need to recall the basics on $\ads$. The three-dimensional anti-de Sitter space $\ads$ is a Lorentzian space form of constant negative sectional curvature, and it can be modeled on $\psl$ with the Killing metric. In this model, the group of space and time-orientation preserving isometries is $\psl^2$, acting by right and left multiplication
$$(A,B) \cdot X = A X B^{-1}.$$ An anti-de Sitter (AdS) manifold is a Lorentzian manifold of constant negative sectional curvature; equivalently, up to scaling the metric, it is one that is locally isometric to $\ads$.

For surface group representations $j,\rho:\Gamma\to \psl$, we define $\Gamma_{j,\rho}<\psl^2$ by \begin{equation}\Gamma_{j,\rho}=\{(j(\gamma),\rho(\gamma)): \gamma \in \Gamma\}.\end{equation}
Kulkarni and Raymond proved that every group acting properly discontinuously on $\ads$ identifies with some $\Gamma_{j,\rho}$, and that (up to switching factors) $j$ must be Fuchsian \cite{Kulkarni_Raymond}. The quotient is a circle bundle over $\mathbb{H}^2/j(\Gamma)$, and the circle fibers are timelike geodesics. It was proved by Salein in \cite{Salein} (if direction) and Kassel in \cite{Kassel} (only if direction) that, when $\Sigma$ is a closed surface of genus at least $2$, $\Gamma_{j,\rho}$ acts properly and discontinuously on $\ads$ if and only if $j$ strictly dominates $\rho$ (see also \cite[Theorem 1.8]{GK_lamination}). For surveys on these results, see \cite[Sections 0.1-0.2]{Tholozan} and \cite[Sections 1.1-1.2]{Almost_domination_Nat}.

In her thesis \cite{Agnese_thesis}, Janigro studied singular anti-de Sitter $3$-manifolds that fiber via timelike geodesics over hyperbolic cone surfaces. Janigro defined the notion of a \emph{spin-cone AdS $3$-manifold}, in which the singularities occur along \emph{timelike geodesics}, generalizing a construction due to Barbot and Meusburger in the context of flat Lorentzian spacetimes, where they modeled massive particles with spin \cite{BS_cone}. Janigro introduced a weak notion of domination for hyperbolic cone surfaces, which she related to spin-cone AdS $3$-manifolds. We extend Janigro's work by showing that domination in the pullback sense gives rise to a singular AdS $3$-manifold (see Theorem~\ref{ads_manifold}). This extension is obtained by adding a singular set to the AdS $3$-manifold constructed by Janigro, thus making the geometric picture more precise. We will see in Section \ref{sec5} that this singular set is closely related to the singular set of $h$, and we note that it also allows for more types of
singularities than just spin-cone singularities. We then specialize to the case where the map $h$ is a branched immersion and analyze the resulting singularities in detail. We arrive at the following theorem.

\begin{principal}\label{principal_ads}
Let $\Sigma$ be a hyperbolic surface with fundamental group $\Gamma$ and let $\rho, j: \Gamma \to \psl$ be two representations. 
Assume that there exist a $\rho$–equivariant map $f$ and a $j$–equivariant \emph{branched immersion} $h$ such that $f^*\sigma < h^*\sigma$. 
Denote by $C$ the singular set of $h$ and $C_0=C/\Gamma$. Then there exists a three–manifold $\fhads$, topologically a solid torus, with the following properties.
\begin{enumerate}
    \item The group $\Gamma_{j,\rho}$ admits a properly discontinuous action on $\fhads$.
    
    \item There exists a continuous map $\mathcal{F} : \fhads \to \widetilde{\Sigma}$ that is equivariant with respect to the action of $\Gamma_{j,\rho}$ on $\fhads$ and the action of $\Gamma$ on $\widetilde{\Sigma}$. Moreover, the fibers of $\mathcal{F}$ are topological circles.
    
    \item  Let $L=\mathcal{F}^{-1}(C)$ and set $\fhadss:=\fhads\setminus L$. 
Then $\fhadss$ is an anti-de Sitter manifold, and the restriction of $\mathcal{F}$ to $\fhadss$ is an $\mathbb{S}^1$-principal bundle over $\Tilde{\Sigma}\backslash C$ with timelike geodesic fibers.

    \item The quotient $\mathcal{M}^{f,h}_* := \fhadss / \Gamma_{j,\rho}$ is a branched anti-de Sitter manifold with singular locus $L/\Gamma_{j,\rho}$. The holonomy representation $\pi_1(\mathcal{M}^{f,h}_*)\to \isom(\ads)$ factors $$\pi_1(\mathcal{M}^{f,h}_*)\to \pi_1(\Sigma\backslash C_0)\to \pi_1(\Sigma)\to \isom(\ads),$$ where the first map is induced by $\mathcal{F}$, the second is induced by inclusion, and the third map is $(j,\rho)$.

\end{enumerate}
\end{principal}

\emph{Branched AdS manifolds} are a special case of spin-cone AdS manifolds (see Definition \ref{defi_spin_cone}). For clarity, in Sections \ref{sec5} and \ref{sec6}, we provide further exposition on Janigro's work in \cite{Agnese_thesis} and explain how our work builds on it. It is worth noting that, in the case where $h$ is a global diffeomorphism, both $\fhads$ and $\fhadss$ coincide with the anti–de Sitter space $\ads$, and we thus recover previously known results about closed AdS $3$-manifolds.

Combining Theorems \ref{principal_A} and \ref{principal_ads}, we obtain the following result, which provides a large class of examples of branched anti-de Sitter $3$-manifolds. 
\setcounter{corx}{1}
\begin{corx}\label{corb}
    Let $X$ be a closed Riemann surface of genus $g\geq 2$ with fundamental group $\Gamma$ and let $\rho:\Gamma\to \psl$ be a representation with $0< \textrm{eu}(\rho)<2g-2$ such that the unique equivariant harmonic map $f:\widetilde{X}\to (\mathbb{H}^2,\sigma)$ is a branched immersion. Then, for every $k$ between $1$ and $2g-2-\textrm{eu}(\rho)$, we can choose a representation $j:\Gamma\to \psl$, $\textrm{eu}(j)=\textrm{eu}(\rho)+k$, with equivariant map $h$ such that the data $(\rho,j,f,h)$  determines a branched anti-de Sitter $3$-manifold as in Theorem \ref{principal_ads}.
\end{corx}
A current conjecture, attributed to Goldman, is that every faithful representation of non-zero and non-extremal Euler class uniformizes a branched hyperbolic structure, or, equivalently, satisfies the hypothesis of the corollary. This conjecture is linked in a circle of ideas around a question of Bowditch as well as Goldman's conjecture about mapping class group dynamics on the $\psl$-character variety (see, for instance, \cite{Faraco+2021+99+108}).

As indicated above, we will see in Theorem \ref{principal_C} that more examples of domination arise than in Theorem \ref{principal_A} (in particular, where $f$ is not a branched immersion but $h$ is), so Corollary \ref{corb} can also be expanded.

\begin{remark}
    If one really wanted to construct a lot of spin-cone AdS $3$-manifolds, then one could use harmonic maps from (universal covers of) punctured surfaces and parabolic Higgs bundles. In this paper, we're interested in the geometry that one can get out of actions of closed surface groups.
\end{remark}

\subsection{Harmonic maps and domination}\label{sec: intro_domination}
Here we give a detailed overview of our main results on harmonic maps and domination. See Sections \ref{sec: hmaps definitions}-\ref{sec: energies} for preliminaries on harmonic maps. 

For motivation, we recall the work of Deroin-Tholozan in \cite{Deroin_thlozan_domination}. Let $X$ be a closed Riemann surface of genus at least $2$ with fundamental group $\Gamma$ and let $\rho:\Gamma\to \psl$ be a non-Fuchsian reductive representation with $\rho$-equivariant harmonic map $f:\widetilde{X}\to \mathbb{H}^2$. Every equivariant harmonic map from
$\widetilde{X}$ determines a holomorphic quadratic differential on $X$ called the \emph{Hopf differential}. By \cite{Wolf_thesis}, there exists a unique Fuchsian representation $j$ together with an equivariant harmonic
diffeomorphism $h:\widetilde{X}\to\mathbb{H}^2$ with the same Hopf differential as $f$. In \cite{Deroin_thlozan_domination}, it is shown that $$f^*\sigma<h^*\sigma.$$ The map $F=f\circ h^{-1}:\mathbb{H}^2\to\mathbb{H}^2$ is $(j,\rho)$-equivariant and strictly contracting, and hence shows that $j$ strictly dominates $\rho$. Tholozan went on to prove that every strictly dominating pair arises in this fashion \cite{Tholozan}. Moreover, the combined works \cite{Deroin_thlozan_domination} and \cite{Tholozan} show the following.
\begin{theorem}[Deroin-Tholozan \cite{Deroin_thlozan_domination} and Tholozan \cite{Tholozan} combined]\label{thm: DT and T}
   Let $j,\rho:\Gamma\to \psl$ be reductive representations with $j$ Fuchsian. Then $j$ strictly dominates $\rho$ if and only if there exists a unique Riemann surface structure $X$ together with equivariant harmonic maps $h$ and $f$ for $j$ and $\rho$ respectively with the same Hopf differential and such that $f^*\sigma<h^*\sigma$.
\end{theorem} 
At this point, Definition \ref{def: domination pullback} is well motivated, and we can refine the problem (\ref{q1}): we look for pairs of equivariant harmonic maps with the same Hopf differential such that the domination inequality holds. In order to do so, we need to understand how harmonic maps with the same Hopf differential can differ. Via the \emph{non-abelian Hodge correspondence}, we establish the following.
\begin{prop}\label{prop: phi and divisors intro} (Proposition \ref{prop: holomorphic data}, loosely stated)
An equivariant harmonic map $f: \widetilde{X}\to \mathbb{H}^2$ (up to translation) is equivalent to a pair $(\phi,D)$ consisting of a holomorphic quadratic differential $\phi$ on $X$, the Hopf differential of $f$, and an effective divisor $D$ on $X$ satisfying certain conditions. When $\phi$ is not zero, $D$ is the divisor of the square root of the holomorphic energy of $f$.    
\end{prop}
See Proposition \ref{prop: holomorphic data} for the precise statement (which includes the case $\phi=0$). When $D=0$, one recovers Wolf's parametrization of Teichm{\"u}ller space \cite{Wolf_thesis}, and when $\phi=0$, one is led to Troyanov's uniformization for hyperbolic branched metrics \cite{Troyanoc}. Proposition \ref{prop: phi and divisors intro} is probably known to some experts (compare with \cite[Theorem 10.8]{Hitchin:1986vp})--see Section \ref{sec: divisors and NAH} for explanation.

Now we can make our problem (\ref{q1}) more precise: for which pairs $(\phi,D_1)$ and $(\phi,D_2)$ do we have domination? The problems turns out to be quite delicate, but we find a complete solution when $(\phi,D_2)$ gives rise to a branched harmonic immersion. Proposition \ref{prop: branched immersions}, although not necessary for the proof, tells us when we have a branched immersion.
\begin{defi}
Let $D_1$ and $D_2$ be divisors on $X$. We write $D_2< D_1$ to mean that $D_2(p)\leq D_1(p)$ for all $p\in X$ and that there exists $q\in X$ such that $D_2(q)<D_1(q)$. 
\end{defi}
\begin{prop}\label{prop: branched immersions}
     An equivariant harmonic map $f:\widetilde{X}\to (\mathbb{H}^2,\sigma)$ associated with a pair $(\phi,D)$, $\phi\neq 0$, is a branched immersion if and only if $2D<(\phi)$ or $(\phi)<2D$.
\end{prop}
When $\phi=0$, the harmonic map is weakly conformal and automatically a branched immersion. The conditions $2D<(\phi)$ and $(\phi)<2D$ correspond to positive and negative Euler numbers respectfully. The key elements of Proposition \ref{prop: branched immersions} are contained in the literature, but the result has never been stated as above. The "only if" direction is 
 observed in \cite[Lemma 3.2]{BBDH}, and the "if" direction follows from \cite[Theorem 4.1, Proposition 6.5]{Gothen_Silva} (which improves \cite[Theorems 3.4 and 4.1]{BBDH}). On route to Theorem \ref{principal_C}, we establish all the tools necessary to prove Proposition \ref{prop: branched immersions}, so for the convenience of the reader we write out a proof. We will also prove a version of Proposition \ref{prop: branched immersions} for surfaces with boundary (see Lemma \ref{lem: modified BBDH}).

 Finally, we state our solution to the domination problem for branched harmonic immersions. The theorem concerns dominating arbitrary maps by branched immersions, and the corollary is about dominating branched immersions themselves.
 \begin{principal}\label{principal_C}
     Let $f,h:\widetilde{X}\to (\mathbb{H}^2,\sigma)$ be equivariant harmonic maps with holomorphic data $(\phi,D_1)$ and $(\phi,D_2)$ respectively, with $\deg D_1, \deg D_2\leq 2g-2$.  Assuming $h$ is a branched immersion, then $$
    f^*\sigma < h^*\sigma \quad \text{if and only if} \quad D_2<D_1 \textrm{ and } D_1+D_2 < (\phi).
    $$
\end{principal}
 \begin{corx}\label{corc}
Let $f,h:\widetilde{X}\to (\mathbb{H}^2,\sigma)$ be equivariant harmonic maps with holomorphic data $(\phi,D_1)$ and $(\phi,D_2)$ respectively, with $\deg  D_1, \deg D_2\leq 2g-2$. Assuming $f$ is a branched immersion, then $$
    f^*\sigma < h^*\sigma \quad \text{if and only if} \quad D_2 < D_1. $$ 
 \end{corx}
 
 The assumption on the degrees is not necessary, but just keeps the statement cleaner. If $\deg D_2\geq 2g-2$, then we can flip back to the case $\deg D_2\leq 2g-2$ by using an outer automorphism of the fundamental group. A harmonic map with data $(\phi,D)$ becomes a map with data $(\phi, (\phi)-D)$, and the pullback metric is unaffected, and thus the condition $D_2<D_1$ becomes $(\phi)-D_2<D_1$. Similar if $\deg D_1\geq 2g-2.$  It is worth noting that if $h$ is not a branched immersion, the condition $D_2<D_1$ does not guarantee domination; see Proposition \ref{prop: counterexample}. 
 
The main result toward the proof of Theorem \ref{principal_C} is Proposition \ref{prop: energy domination}, which is of independent interest. It provides an inequality between functions satisfying a \emph{Bochner-type equation}, generalizing the key inequality from \cite[Section 1.8]{Schoen_Yau_book} and \cite[Lemma 2.6]{Deroin_thlozan_domination}. This proposition is also a key step in establishing Proposition \ref{prop: branched immersions}.
\begin{remark}
    Given harmonic maps $f,h:\widetilde{X}\to (\mathbb{H}^2,\sigma)$ with the same Hopf differential, the map $F=(f,h): \widetilde{X}\to (\mathbb{H}^2\times \mathbb{H}^2,\sigma\oplus (-\sigma))$ determines a maximal surface (zero mean curvature at immersed points). The domination condition $f^*\sigma<h^*\sigma$ says that $F$ is a spacelike immersion off the singular set of $h$.
\end{remark}

\begin{remark}
In the works \cite{Deroin_thlozan_domination}, \cite{Tholozan}, \cite{Nathaniel_JDG}, \cite{Almost_domination_Nat}, the main domain results concern pairs $(\rho,j),$ where $\rho$ is an action by isometries on a $\textrm{CAT}(-1)$ Hadamard manifold $M$ and $j$ is a Fuchsian representation to $\psl.$ It should be possible to prove extensions of Theorems \ref{principal_A} and \ref{principal_C} for such pairs. Toward this, the key observations are that, for maps to $M$, there is a notion of holomorphic energy (see \cite[Section 2.2]{Deroin_thlozan_domination}) and we have a Bochner formula (see \cite[Lemma 2.3]{Deroin_thlozan_domination}).
\end{remark}

With our application to anti-de Sitter geometry in mind, we are mainly concerned with dominating by branched immersions, which puts us in the case $|\textrm{eu}(\rho)|<|\textrm{eu}(j)|$. That being said, after Theorem \ref{principal_C}, it is natural to inquire about the equality case $|\textrm{eu}(\rho)|=|\textrm{eu}(j)|.$ Since it's not focused toward an application in this paper, we don't carry out a general treatment, but we just give examples that point to future directions. 
\begin{principal}\label{principal_D}
For all even $g\geq 2$, there exists a Riemann surface $Y$ with fundamental group $\Gamma$ such that the following holds.
    \begin{enumerate}
        \item There exist representations $\rho,j:\Gamma\to \psl$ with equivariant harmonic maps $f,h:\widetilde{Y}\to (\mathbb{H}^2,\sigma)$ respectively, such that  $\textrm{eu}(\rho)=\textrm{eu}(j)=0$ and $f^*\sigma< h^*\sigma$.
        \item There exists a representation $\rho$ with equivariant harmonic map $f$ with the property that there exist no pairs $(j,h)$ consisting of a non-Fuchsian representation $j$ and an equivariant harmonic map $h$ with the same Hopf differential such that $f^*\sigma< h^*\sigma$.
    \end{enumerate}
\end{principal}
In (2), the underlying representation has Euler number zero.

\subsection{Future directions} We list a few research directions that follow this work.

\subsubsection{Length spectrum domination}
In line with \cite{GKW_folded}, it is natural to study pairs $\rho,j:\Gamma\to \psl$ such that, for all $\gamma\in \Gamma,$ $\ell(\rho(\gamma))\leq \ell(j(\gamma))$, where $\ell(\cdot)$ is the hyperbolic translation length. It would be interesting to compare this length spectrum domination with the domination considered here.
\subsubsection{Higher rank}
Our domination problems generalize easily for representations to a higher rank non-compact semisimple Lie group $G$, with the space $\mathbb{H}^2$ replaced with a Riemannian symmetric space of $G$. 
For generalizations of \cite{Deroin_thlozan_domination}, with Fuchsian representations replaced with Hitchin representations to $G=\textrm{SL}(n,\R)$ see \cite[Theorem 1.8 and Conjecture 1.11]{Dai_Li_cyclic} and \cite{Dai_Li_Domination}. Extending our results in this paper seems approachable for harmonic maps arising from Coxeter cyclic $G$-Higgs bundles, which are equivalent to solutions to affine Toda equations (see \cite{Nathaniel_ognjen}). There is even a Toledo number for cyclic $G$-Higgs bundles, which generalizes the Euler number (see \cite{GP_G}). 

\subsubsection{More on singular $\textrm{AdS}$ $3$-manifolds}

Our constructions in Section \ref{sec5} (in particular, Theorem \ref{ads_manifold}) allow the construction of general singular AdS $3$-manifolds from dominating pairs, rather than only branched AdS $3$-manifolds arising from branched immersions. By developing more exotic examples of singular AdS $3$-manifolds, one may gain further insight into the nature of singularities in anti--de Sitter geometry.

\subsection{Outline of the paper}
In Sections \ref{sec2} and \ref{sec3}, we study the domination problem for harmonic maps, while in Sections \ref{sec4}, \ref{sec5}, and \ref{sec6}, we study anti-de Sitter geometry and construct singular anti-de Sitter $3$-manifolds from dominating pairs. Sections \ref{sec2}-\ref{sec3} and \ref{sec4}-\ref{sec6} can essentially be read independently. 

In more detail, Section \ref{sec2} provides preliminaries on harmonic maps and carefully states and proves Proposition \ref{prop: holomorphic data}, which characterizes harmonic maps via holomorphic data $(\phi,D)$. In Section \ref{sec3}, we address the domination problem and prove Proposition \ref{prop: branched immersions} as well as Theorems \ref{principal_A}, \ref{principal_C}, and \ref{principal_D}. In Section \ref{sec4}, we introduce spin-cone AdS $3$-manifolds, and in Section \ref{sec5}, we construct singular AdS $3$-manifolds from general dominating pairs (Theorem \ref{ads_manifold}). Finally, in Section \ref{sec6}, we specialize to branched immersions and show that, in this case, the examples from Section \ref{sec5} are branched AdS $3$-manifolds, thereby proving Theorem \ref{principal_ads}.

\subsection{Acknowledgements}
This project stems from discussions with Andrea Seppi, to whom we are very grateful, and to whom we thank further for helpful comments on the first draft. We also thank Francesco Bonsante for sharing the thesis \cite{Agnese_thesis}.

\section{Preliminaries on harmonic maps}\label{sec2}
\subsection{Equivariant harmonic maps}\label{sec: hmaps definitions}
Let $\Sigma_g$ be a closed oriented surface of genus at least $2$ and let $X$ be a Riemann surface structure on $\Sigma_g$ with universal cover $\Tilde{X}$. Let $(M,\sigma)$ be a Riemannian manifold and let $f:\widetilde{X}\to M$ be a $C^2$ map. The derivative $df$ defines a section of the 
endomorphism bundle $T^*\widetilde{X}\otimes f^*TM$. Complexifying $T^*\widetilde{X}\otimes f^*TM$, let $df=\partial f +\overline{\partial}f$ be the decomposition into $(1,0)$ and $(0,1)$ parts. We denote by $\nabla$ the connection on $f^*TM\otimes \mathbb{C}$ induced by the Levi-Civita connection of $\sigma$, and we also use $\nabla$ for its extension to $f^*TM\otimes \mathbb{C}$-valued forms. The definition below depends on the Riemann surface structure $X$ and the metric $\sigma.$ 
\begin{defi}
The map $f$ is harmonic if $\nabla^{0,1}\partial f=0.$
\end{defi}
In this paper, we are concerned with equivariant harmonic maps from $\widetilde{X}$ to the space $(\mathbb{H}^2, \sigma)$, where $\mathbb{H}^2$ is the $2$-dimensional upper half-plane and $\sigma$ is a hyperbolic metric of constant curvature $-1$. The starting point is the existence theorem of Donaldson \cite{Donaldson} and Corlette \cite{Corlette}. A representation $\rho:\Gamma\to \psl$ is irreducible if it is not contained in a parabolic subgroup, and $\rho$ is reductive if the Zariski closure of $\rho(\Gamma)$ is a reductive group. Geometrically, and perhaps more intuitively, $\rho$ is irreducible if the induced action of $\Gamma$ on the Gromov boundary $\partial_\infty \mathbb{H}$ has no global fixed point, and $\rho$ is reductive if $\rho(\Gamma)$ preserves a geodesic in $\mathbb{H}^2.$
\begin{theorem}[Donaldson, Corlette]
    Let $X$ be a Riemann surface structure on $\Sigma_g$ and let $\rho:\Gamma\to \psl$ be a representation, acting by isometries on $(\mathbb{H}^2,\sigma).$ Then there exists a $\rho$-equivariant harmonic map $f:\widetilde{X}\to (\mathbb{H}^2,\sigma)$ if and only if $\rho$ is reductive. 
\end{theorem}
By Sampson's argument in \cite[Theorem 3]{Sampson}, if $\rho$ is irreducible, then $f$ is unique. When $\rho$ is just reductive and not irreducible, again by \cite[Theorem 3]{Sampson}, there is a $1$-parameter family of harmonic maps, each of which is a parametrization onto the invariant geodesic, and all of the harmonic maps are related by an isometric translation along the geodesic. Since the harmonic maps are related by isometries, the pullback metric never depends on the choice of harmonic map.

\subsection{Energies and Hopf differentials}\label{sec: energies}
    Given a harmonic map from a Riemann surface, one can associate a number of analytic objects. For a detailed reference, see \cite[Section 1-2]{Schoen_Yau_book}. To begin, let $X$ be a Riemann surface as above and fix a metric $\nu$ on $X$ that's compatible with the Riemann surface structure. We use $\nu$ as well to denote the lift to a metric on the universal cover $\widetilde{X}$.

Keeping things general for now, let $\rho:\Gamma\to \textrm{Isom}(M,\sigma)$ be an action by isometries and let $f:\widetilde{X}\to (M, \sigma)$ be a $\rho$-equivariant $C^2$ map. The (possibly degenerate) pullback metric $f^*\sigma$ extends bilinearly to the complexified tangent bundle of $\widetilde{X}$ and then decomposes into types as 
\begin{equation}\label{eq: pullback metric}
    f^*\sigma = 2\sigma(\partial f,\overline{\partial} f) + \sigma(\partial f, \partial f)+ \sigma(\overline{\partial}f,\overline{\partial} f).
\end{equation}
We write $2\sigma(\partial f,\overline{\partial} f )=e(f)\nu,$ where $e(f)$ is called the energy density function. As well, the quadratic differential $\phi(f):=\sigma(\partial f, \partial f)$ is called the \emph{Hopf differential} of $f$ (note also that $\overline{\phi}(f)=\sigma(\overline{\partial}f,\overline{\partial} f)$). Since $f$ is $\rho$-equivariant, $e(f),\phi(f),$ $\overline{\phi}(f)$, and $f^*\sigma$ are invariant under the action of $\Gamma$ and hence descend to $X$.
The formula (\ref{eq: pullback metric}) is rewritten as
\begin{equation}\label{eq: pullback metric2}
f^*\sigma = e(f)\nu + \phi(f) + \overline{\phi}(f).
\end{equation}
The significance of $e(f)$ stems from the fact that the equation $\nabla^{0,1}\partial f=0$ can be seen as the Euler-Lagrange equation for the Dirichlet energy \begin{equation*}
    \mathcal{E}(X,f) = \int_{\Sigma_g} e(f)dA_\nu,
\end{equation*}
where $dA_\nu$ is the area form of $\nu$ (note $\mathcal{E}(X,f)$ does not depend on the choice of $\nu$). That is, harmonic maps are equivalently critical points of $\mathcal{E}(X,\cdot)$. As for the Hopf differential $\phi(f)$, it is holomorphic when $f$ is harmonic, and when the target has dimension at most $2,$ so for example when $M=\mathbb{H}^2,$ the converse holds as well.

We now set $M=\mathbb{H}^2$. In this special case, there are more analytic quantities to probe the harmonic map. Working in local coordinates, where $\sigma=\sigma(w)|dw|^2$, $\nu=\nu(z)|dz|^2$, we set $$H(f)=\frac{\sigma(f(z))}{\nu(z)}\Big |\frac{\partial f}{\partial z}\Big |^2, \hspace{1mm} L(f)=\frac{\sigma(f(z))}{\nu(z)}\Big |\frac{\partial f}{\partial \overline{z}}\Big |^2.$$ The function $H(f)$ is called the \emph{holomorphic energy} and $L(f)$ is called the \emph{anti-holomorphic energy}. They satisfy $$e(f)=H(f)+L(f), \hspace{1mm} |\phi(f)|^2\nu^{-2}=H(f)L(f),$$ and the vanishing of the Jacobian of $f$ is equivalent to the vanishing of the function $$J(f):=H(f)-L(f).$$ When $f$ is harmonic, the functions $H(f)$ and $L(f)$ are either identically zero or have isolated zeros and satisfy the Bochner formulae: away from the zeros of $H(f),$ 
\begin{equation}\label{eq: Bochner for H}
    \frac{1}{2}\Delta_\nu \log H(f)= -\kappa_\sigma(H(f)-L(f))+\kappa_\nu,
\end{equation}
where $\kappa_\sigma$ and $\kappa_\nu$ are the sectional curvatures of $\sigma$ and $\nu$, respectively (so $\kappa_\sigma=-1$), and away from the zeros of $L(f),$ 
$$\frac{1}{2}\Delta_\nu \log L(f)= -\kappa_\sigma(L(f)-H(f))+\kappa_\nu$$ (see \cite[Section 1.7]{Schoen_Yau_book}).
Note that the Bochner formula for $L(f)$ can be seen as a consequence of the Bochner formula for $H(f)$ and the equation $|\phi(f)|^2\nu^{-2}=H(f)L(f)$. Note that for a non-irreducible representation, again using that harmonic maps are related by isometries, $H(f)$ is independent of the choice of harmonic map.

We end this subsection with a discussion on Euler numbers. The \emph{Euler number} of a representation $\rho,$ denoted $\textrm{eu}(\rho),$ can be defined in many ways; for instance, using characteristic classes, or as an obstruction to lifting $\rho$ to $\widetilde{\psl}$. Here, we give the most naive definition: 
\begin{equation}\label{eq: Euler number}
    \textrm{eu}(\rho):=\frac{1}{2\pi}\int_{\Sigma_g} dA_{f^*\sigma},
\end{equation}
where $f:\widetilde{\Sigma}_g\to \mathbb{H}^2$ is any $\rho$-equivariant $C^2$ map, and $dA_{f^*\sigma}$ is the area form of $f^*\sigma$ (which might be degenerate). By an application of Stoke's theorem, the integral above indeed depends only on $\rho$ and not on $f$.
It is also useful to recall the equality of $2$-forms 
\begin{equation}\label{eq: 2-forms}
 dA_{f^*\sigma}=J(f)dA_\nu.
\end{equation}
The following lemma allows us to make Definition \ref{def: div rho} below.
\begin{lemma}\label{lem: H not zero}
     Let $X$ be a Riemann surface structure on $\Sigma_g$ and let $\rho:\Gamma\to \psl$ be a reductive representation with equivariant harmonic map $f:\widetilde{X}\to (\mathbb{H}^2,\sigma)$. If $\textrm{eu}(\rho)\geq 0$, then $H(f)$ cannot be identically zero.
\end{lemma}
\begin{proof}
    By (\ref{eq: Euler number}), (\ref{eq: 2-forms}), and $J(f)=H(f)-L(f)$, we have 
    \begin{equation}\label{eq: split_Jacobian}
        \frac{1}{2\pi}\int_{\Sigma_g} H(f) dA_\nu= \textrm{eu}(\rho)+\frac{1}{2\pi}\int_{\Sigma_g} L(f) dA_\nu.
    \end{equation}
 If $\textrm{eu}(\rho)>0$ then the right hand side above is strict, so $H(f)$ has to be positive somewhere. If $\textrm{eu}(\rho)=0$ and $H(f)=0$ identically, then the right hand side shows that $f$ is a constant map, which is impossible. 
\end{proof}
\begin{defi}\label{def: div rho}
    Let $X$ be a Riemann surface structure on $\Sigma_g$ and let $\rho:\Gamma\to \psl$ be a representation with $\textrm{eu}(\rho)\geq 0$ and equivariant harmonic map $f:\widetilde{X}\to (\mathbb{H}^2,\sigma)$. The divisor of $(X,\rho)$, denoted $D_X(\rho),$ is the divisor of the square root of the holomorphic energy of $f$. 
\end{defi}
Starting from the characterization (\ref{eq: Euler number}), a classical argument, which involves the Bochner formula, (\ref{eq: 2-forms}), and the Gauss-Bonnet theorem, can be applied nearly verbatim to prove the following. The argument can be found in \cite[pp. 11-12]{Schoen_Yau_book}.
\begin{prop}\label{prop: euler number}
     Let $\rho:\Gamma\to \psl$ be reductive with $\textrm{eu}(\rho)\geq 0$. Then $\textrm{eu}(\rho)=2g-2-\deg D_X(\rho).$
\end{prop}
    If we reverse the orientation of $\Sigma_g$ or precompose $\rho$ with an outer automorphism of $\psl$, then the Euler numbers flip sign. For this reason, we often restrict to the case $\textrm{eu}(\rho)\geq 0$. For $\textrm{eu}(\rho)\leq 0$, one has results analogous to above, with $H(f)$ replaced with $L(f).$

\subsection{Divisors and Higgs bundles}\label{sec: divisors and NAH}

The purpose of this section is to prove Proposition \ref{prop: holomorphic data} below (or, Proposition \ref{prop: phi and divisors intro}). We use Higgs bundles and the non-abelian Hodge correspondence; Higgs bundles will not come up again, so the unfamiliar reader might benefit from skipping the proofs on a first reading.

For a non-zero holomorphic section $\alpha$ of a holomorphic line bundle on $X$, we denote the divisor by $(\alpha).$ Let $\mathcal{K}$ be the canonical bundle of $X,$ so that $H^0(X,\mathcal{K}^2)$ is the space of holomorphic quadratic differentials on $X$.
\begin{defi}
    Let $\mathcal{D}$ be the set of pairs $(\phi,D),$ where $D$ is an effective divisor on $X$ with $\deg  D\leq 2g-2$ and $\phi\in H^0(X,\mathcal{K}^2)$, such that, if $\phi\neq 0$, then $D\leq (\phi)$ (as functions). We add the further condition that if $\deg  D=2g-2,$ then $\phi\neq 0$.
\end{defi}

\begin{prop}\label{prop: holomorphic data}
The set $\mathcal{D}$ is in bijection with the space of conjugacy classes of reductive representations from $\Gamma$ to $\psl$ with non-negative Euler class. If $(\phi,D)$ is associated with $\rho:\Gamma\to \psl$, then $D=D_X(\rho)$ and $\phi$ is the Hopf differential of any $\rho$-equivariant harmonic map.
\end{prop}
To parametrize representations with non-positive Euler numbers, one takes pairs $(\phi,D)$ such that $\deg D\geq 2g-2$ and $D\geq (\phi)$; simply use an outer automorphism. About the case $\deg D = 2g-2$, the equation (\ref{eq: split_Jacobian}) shows that if $\deg D = 2g-2$, i.e, $\textrm{eu}(\rho)=0$, then no reductive representation can carry a harmonic map $f$ with $\phi(f)=0$, for then we would have $H(f)=L(f)=0$.

Proposition \ref{prop: holomorphic data} should be known to experts.  
The parametrization of the character variety by pairs $(\phi,D)$ is proved in Hitchin's original paper \cite[Theorem 10.8]{Hitchin:1986vp} (outside of Euler number $0$, although comments are made on that case), but the characterization in terms of zeros of the holomorphic energy does not appear to be recorded. We essentially redo Hitchin's proof in a different language, and explain, from our point of view, how $D$ comes from the holomorphic energy.

\subsubsection{$G$-Higgs bundles}
Since we're considering the adjoint group $\psl$, we work with $G$-Higgs bundles. Representations with even Euler class lift to $\textrm{SL}(2,\R),$ and so for those representations one could use linear Higgs bundles. In general, one could transfer to the linear setting using the isomorphism $\psl\simeq \textrm{SO}_0(1,2)$, but the proofs are a bit faster and more natural in the principal bundle setting.

There are many excellent sources on $G$-Higgs bundles, and we don't need to recall everything here. We mostly draw on \cite{Nathaniel_ognjen}, which we refer to for more details. Let $G$ be a non-compact simple complex Lie group with maximal compact subgroup $K$. The space $G/K,$ equipped with the metric induced by the Killing form on the Lie algebra of $G,$ is a Riemannian symmetric space of non-compact type. For $G=\psl$, this symmetric space is $\mathbb{H}^2$. Let $\mathfrak{g}$ and $\mathfrak{k}$ be the Lie algebras of $G$ and $K$ respectively, and let $\mathfrak{p}$ be the Killing orthogonal complement of $\mathfrak{k}$ in $\mathfrak{g}$. We write $G^{\C},K^{\C},\mathfrak{g}^{\C},$ etc., for complexifications.
\begin{defi}
    A $G$-Higgs bundle over a Riemann surface $X$ is a pair $(P,\Phi),$ where $P$ is a principal $K^{\C}$-bundle and $\Phi$ is a holomorphic $1$-form valued in the associated bundle $P\times_{\textrm{Ad}|_{K^{\C}}} \times \mathfrak{p}^{\C}$ called the Higgs field.
\end{defi}
See \cite[Section 2.2]{Nathaniel_ognjen} for more details on the discussion below. As is well known, an equivariant harmonic map $f$ to $G/K$ gives rise to a $G$-Higgs bundle. Very briefly, $G\to G/K$ is a principal $K$-bundle that carries a principal connection induced from the Maurer-Cartan form on $G$. The map $f$ pulls back a $K$-bundle $Q$ over $X$ with a principal connection. Using the Maurer-Cartan isomorphism, the derivative of $f$ identifies as a $1$-form valued in $Q\times_{\textrm{Ad}|_{K}} \mathfrak{p}$, say $\psi$. Upon complexifying $Q$ to a principal $K^{\C}$-bundle $P$, we can split $\psi$ into types as $\psi=\psi^{1,0}+\psi^{0,1}$, and the harmonicity of $f$ is equivalent to the assertion that, with respect to the Koszul-Malgrange holomorphic structure on $P\to X$ associated with the principal connection on $Q$, $\psi^{1,0}$ is a Higgs field. Given a $G$-Higgs bundle, it arises from a harmonic map, in a way that undoes the procedure above, if and only if one can solve Hitchin's self-duality equations (see \cite[Definition 2.3]{Nathaniel_ognjen}).

\subsubsection{Proposition \ref{prop: holomorphic data}}
As explained in \cite[Section 3.5]{Nathaniel_ognjen}, a $\psl$-Higgs bundle is equivalent to a $\textrm{PSL}(2,\C)$-Higgs bundle equipped with a holomorphic gauge transformation $s$ of $P$ satisfying certain properties and such that $s^*\phi = -\phi$. In the language of \cite{Nathaniel_ognjen}, $(P,\phi)$ is Coxeter cyclic. By \cite[Proposition 1.2]{Nathaniel_ognjen}, the Higgs bundle is equivalent to the data of a line bundle $L$ and sections $\alpha$ and $\beta$ of $L\otimes \mathcal{K}$ and $L^{-1}\otimes \mathcal{K}$ respectively. An argument of Hitchin from \cite[Section 10]{Hitchin:1986vp} shows that we can always specify things so that, when the $G$-Higgs bundle arises from an equivariant harmonic map, then $\deg L$ is the Euler number of the underlying representation. In this case, the Hopf differential, up to dividing by a positive scalar $c$, is the product $\alpha\beta:=\alpha\otimes \beta$. From the proof of \cite[Proposition 1.2]{Nathaniel_ognjen}, under this constraint, two $G$-Higgs bundles giving triples $(L,\alpha,\beta)$ and $(L',\alpha',\beta')$ as above are isomorphic if and only if there exists an isomorphism from $L\to L'$ that intertwines $\alpha$ with $\alpha'$ and $\beta$ with $\beta'$.

From \cite[Theorems A and 4.3]{Nathaniel_ognjen}, once we've fixed a conformal metric $\nu$ on $X$, a solution to Hitchin's self-duality equations is equivalent to a Hermitian metric $\mu$ on $L$ (with dual metric $\mu^{-1}$ on $L^{-1}$) solving the equation, for functions $e_\alpha=\mu(\alpha,\alpha)\nu^{-2}$ and $e_\beta=\mu^{-1}(\beta,\beta)\nu^{-2}$ on $X$, away from their zeros,
\begin{equation}\label{eq: toda}
    \frac{1}{2}\Delta_\nu \log e_\alpha = e_\alpha - e_\beta + \kappa_\nu, \hspace{1mm} \frac{1}{2}\Delta_\nu \log e_\beta = e_\beta - e_\alpha + \kappa_\nu.
\end{equation}
It is also shown in \cite{Nathaniel_ognjen} that $e_\alpha e_\beta= |\phi|^2\nu^{-2}$, where $\phi$ is the Hopf differential of the harmonic map, and that $e_\alpha+e_\beta$ is the energy density of the harmonic map. There is one subtlety: distinct solutions to Hitchin's equations could produce the same solution to (\ref{eq: toda}). Also, note that \cite[Theorems A]{Nathaniel_ognjen} concerns solutions to Hitchin's equations for stable $G$-Higgs bundles, but the general case follows using \cite[Remark 3.17]{Nathaniel_ognjen} (it is exactly this latter case in which we can have multiple solutions). From the relation $e_\alpha e_\beta= |\phi|^2\nu^{-2}$, solving (\ref{eq: toda}) requires solving only for $e_\alpha$, and the equations reduce to the Bochner formula (\ref{eq: Bochner for H}). 

The equations (\ref{eq: toda}) are a basic example of affine Toda equations. Existence and uniqueness results for affine Toda equations were recently established in \cite[Theorem 1.3]{McIntosh}, which shows in our context (the very simplest case of \cite[Theorem 1.3]{McIntosh}) that, given $(L,\alpha,\beta)$, if $\alpha,\beta\neq 0$, then (\ref{eq: toda}) has a unique solution for $e_\alpha$ with a prescribed vanishing divisor. By virtue of $L\otimes \mathcal{K}$ and $L^{-1}\otimes\mathcal{K}$ having non-vanishing sections, it's implicit in this case that $2-2g\leq \deg L\leq 2g-2$. If either of $\alpha$ or $\beta$ is equal to zero, \cite[Theorem 1.3]{McIntosh} shows that one has a solution if and only if
$2-2g\leq \deg L\leq 2g-2$ and $\deg L\neq \deg  L^{-1}$ ($\deg L =\deg L^{-1}$ can occur only if the common degree is $0$) and moreover the solution is unique. We prove the following.

\begin{prop}\label{prop: deg L euler number}
    Assume that $\deg L> 0$ and that (\ref{eq: toda}) admits a solution with associated representation $\rho$ and harmonic map $f$. Then $H(f)=e_\beta$ and $L(f)=e_\alpha$. 
\end{prop}
\begin{proof}
Let $e(f)$ and $\phi(f)$ be the energy and Hopf differential of $f$ respectively. Consider the function, on $X\times \C$, 
$$p(z,\lambda)=\lambda^2 - e(f)(z)\lambda + (|\phi(f)|^2\nu^{-2})(z).$$
Over each point $z$ of $X$, $\{e_\alpha(z),e_\beta(z)\}$ and $\{H(f)(z),L(f)(z)\}$ are the zeros of $p(z,\cdot)$. Since all of the functions in question are real analytic (for they solve a semi-linear elliptic PDE with real analytic coefficients), it follows that, as sets of functions on $X$, $\{e_\alpha,e_\beta\}=\{H(f),L(f)\}.$ Since $\deg L>0,$ so that $\deg L\neq \deg L^{-1}$, it follows that $\deg(e_{\alpha}) \neq \deg(e_{\beta})$, for their divisors agree with that of $\alpha$ and $\beta$ respectively. By Proposition \ref{prop: euler number}, $\deg(e_{\beta})$ captures the vanishing divisor of $H(f)$, and hence $H(f)=e_\beta$ and $L(f)=e_\alpha.$
\end{proof}
Note that the proof used $\deg L\neq 0$ only in the last line. If $\deg L=0$, then, as unordered sets of functions, $\{e_{\alpha},e_{\beta}\}=\{H(f),L(f)\}$.

\begin{remark}
    We did not strictly need to introduce $H(f)$ and $L(f)$, nor their Bochner formulae. Indeed, we could have just defined $D_X(\rho)$ using $e_\alpha$ and $e_\beta$. We prefer to use $H(f)$ and $L(f)$ and to show the equality with $e_\alpha$ and $e_\beta$ respectively because these are classical functions and carry geometric meaning.
\end{remark}

\begin{lemma}\label{lem: bundle data injective}
Let $\rho$ be a reductive representation with $\textrm{eu}(\rho)\geq 0$ and carrying equivariant harmonic map $f$, giving rise to data $(c\alpha\beta,(\beta))$. Then the conjugacy class of $\rho$ is determined by $(\alpha\beta,(\beta))$.
\end{lemma}
Recall that the $c$ above is the constant chosen so that $\phi=c\alpha\beta$. In making the statement we used that for $\textrm{eu}(\rho)\geq 0$, $\beta$ does not vanish identically. In the proof below, when unspecified, a product of sections is the tensor product. 
\begin{proof}
    We say that $(L,\alpha,\beta)$ and $(L',\alpha',\beta')$ are isomorphic if the corresponding $G$-Higgs bundles are isomorphic. When we can solve the self-duality equations, this is equivalent to the associated representations being conjugate. Thus, we only need to show that the isomorphism class of $(L,\alpha,\beta)$ is determined by $(\alpha\beta,(\beta))$.

Setting $\varphi(L,\alpha,\beta)=(\phi,D)$, with $\phi=c\alpha\beta$ and $D=(\beta)$, the section $\beta$ defines a $\C^*$-family of isomorphisms from $L^{-1}\otimes \mathcal{K}\otimes \mathcal{O}(-D)$ to the trivial bundle $\mathcal{O}$; each isomorphism takes $\beta$ to a constant in $\C^*$, and specifying that constant determines the isomorphism. We choose the isomorphism so that $\beta$ is sent to $1$.
There is an induced isomorphism from $L$ to $\mathcal{K}\otimes \mathcal{O}(D)$. Under the dual isomorphism from $L^{-1}$ to $\mathcal{K}^{-1}\otimes \mathcal{O}(D)$, since $c\alpha\beta=\phi,$ $\alpha$ becomes $c^{-1}\phi \beta^{-1}.$ 
\end{proof}

With preliminaries established, we can now prove Proposition \ref{prop: holomorphic data}. 
\begin{proof}[Proof of Proposition \ref{prop: holomorphic data}]
   Let $\mathcal{L}$ be the set of isomorphism classes of triples $(L,\alpha,\beta)$ such that we can solve (\ref{eq: toda}) (uniquely) and let $\mathcal{L}_k$ be the subset of $\mathcal{L}$ such that $L$ has degree $k$. For $k\geq 0$, we define $\varphi_k:\mathcal{L}_k\to \mathcal{D}$ by $\varphi_k([L,\alpha,\beta])=(\phi,D),$ where $\phi=c\alpha\beta$ and $D$ is the divisor of the square root of the holomorphic energy of the harmonic map. For $k>0$, this is just $(\beta)$, but for $k=0$ we're not quite able to distinguish. Each $\varphi_k$ indeed lands in $\mathcal{D}$: for $k>0$, Proposition \ref{prop: deg L euler number} gives $H(f)=e_\beta$, and hence $\deg D =\deg L^{-1}+2g-2= 2g-2-k$. Still in the case $k>0$, the condition $D\leq (\phi)$ is obvious. For $k=0$, we can't pick out whether $H(f)=e_\beta$ or $e_\alpha$, but the existence and uniqueness theory in this case shows that neither function vanishes identically, and as well using that $\deg L=\deg L^{-1}=0,$ we get that $\deg D=2g-2.$ We prove the proposition by showing that, for every $k\geq 0$, $\varphi_k$ is a bijection onto the set $\{(\phi,D)\in \mathcal{D}:\deg  D = 2g-2-k\}.$
   
   By Lemma \ref{lem: bundle data injective}, each $\varphi_k$ is injective. There is no issue for $k=0$: none of $\alpha$, $\beta$, or $\phi$ are zero sections, so we can express $D=(\beta)$ or $D=(\alpha\beta)-(\beta).$ For surjectivity, fix a pair $(\phi,D)$ with $\deg D=2g-2-k$. Looking for a $\varphi_k$-preimage $(L,\alpha,\beta),$ we take $L=\mathcal{K}^{-1}\otimes \mathcal{O}(D)$. Note that, via the inclusion $\mathcal{O}\to \mathcal{O}(D)$, the constant section $1$ of $\mathcal{O}$ determines a canonical section of $\mathcal{O}(D).$ We take $\alpha$ to be this section and define $\beta$ by $\beta=c^{-1}\phi\alpha^{-1}.$ Then $\varphi_k([L,\alpha,\beta])=(\phi,D),$ as desired.
\end{proof}

\section{Harmonic maps and domination}\label{sec3}
In this section, we prove Theorems \ref{principal_A}, \ref{principal_C}, and \ref{principal_D}, and Corollary \ref{corc}. Throughout, let $X$ be a closed Riemann surface of genus $g$ at least $2$ and let $\nu$ be a conformal metric on $X$.

\begin{subsection}{Domination inequality}
Here we establish the key analytic input toward our main results, Proposition \ref{prop: energy domination} below. Proposition \ref{prop: energy domination}, interesting in its own right, generalizes the most important inequality from \cite{Deroin_thlozan_domination}, namely, \cite[Lemma 2.6]{Deroin_thlozan_domination} (whose proof generalizes a classical argument as found in \cite[Section 1.8]{Schoen_Yau_book}). See also \cite[Lemma 4.3]{Gothen_Silva}.

  Let $\phi$ be a holomorphic quadratic differential on $X$. 
  \begin{defi}
       We say that a function $u$ on the complement of a discrete subset of $X$ is a Bochner solution for $\phi$ if $u=\log H(f)$ or $u=\log L(f)$ for some equivariant harmonic map $f$ with Hopf differential $\phi$.
  \end{defi}
Equivalently, $u$ is a Bochner solution if there is a divisor $D$ satisfying certain conditions (for example, dominated by $(\phi)$ if $\phi\neq 0$) such that on the complement of the support of $D$, $u$ is defined, $C^2$, and solves 
\begin{equation}\label{eq: Bochner solution}
    \frac{1}{2}\Delta_\nu u = e^{u} - |\phi|_\nu^2 e^{-u} +\kappa_\nu,
\end{equation}
where $|\phi|_\nu^2=|\phi|^2\nu^{-2}$. Moreover, at a point $p$ in the support of $D$, $u$ is asymptotic to $D(p)\textrm{log}|z|$. 
\begin{prop}\label{prop: energy domination}
   Let $u_1$ and $u_2$ be Bochner solutions for $\phi$ with divisors $D_1$ and $D_2$ respectively. If $D_2<D_1$, then $u_1<u_2$ on the complement of the support of $D_2$.
 \end{prop}
\begin{proof}
    We first show $u_1\leq u_2$, and then we promote the result to $u_1<u_2$ outside of the support of $D_2$. Set $u:=u_1-u_2$. By our assumptions, $u$ is bounded above and tends to $-\infty$ on a non-empty discrete subset of $X$.
    Assume for the sake of contradiction that $u>0$ at a point. Then the open subset $U=\{z\in X: u(z)>0\}$ is non-empty. Since $u$ tends to $-\infty$ somewhere, $U$ is a proper open subset of $X$. Taking the Laplacian of $u$, (\ref{eq: Bochner solution}) yields
\begin{equation}\label{eq: diff eq for u}
    \frac{1}{2}\Delta_\nu u = (e^{u_1}-e^{u_2}) -|\phi|_\nu^2(e^{-u_1}-e^{-u_2}) = e^{u_2}(e^u -1) -e^{-u_2}|\phi|_\nu^2(e^{-u}-1).
\end{equation}
Since $\partial U$ is just the zero set of $u$, $u$ is continuous on $\overline{U}$. By \eqref{eq: diff eq for u}, $u|_{\overline{U}}$ is subharmonic on $U$. By the weak maximum principle, $u|_{\overline{U}}$ is maximized on $\partial U$. But this contradicts $u|_{\partial U}=0$. We deduce that $u\leq 0$.

We now prove that the inequality is strict. The strict inequality clearly holds near the support of $D_1-D_2$. On the complement of the support of $D_1$, which we will call $V$, we have $$\frac{1}{2}\Delta_\nu u = \Big (1+e^{-u_1-u_2}|\phi|_\nu^2\Big )(e^{u_1}-e^{u_2}) = (e^{u_2}+e^{-u_1}|\phi|_\nu^2)(e^{u}-1).$$ Since $e^x\geq x+1$ and $u\leq 0$, 
$$\Delta_\nu u \geq K u,$$ where $K=2\textrm{max}_X(e^{u_2}+e^{-u_1}|\phi|_\nu^2)).$ Hence, by a consequence of the strong maximum principle \cite{Minda}, either $u=0$ or $u<0$ on all of $V$. Since $u$ tends to $-\infty$ as we approach the support of $D_1-D_2$, the former cannot occur. We conclude that $u_2<u_1$ on the set in question. 
\end{proof}
\begin{remark}\label{rem: boundary case}
   Proposition \ref{prop: energy domination} extends easily to the case of a closed surface $\Sigma$ with compact boundary $\partial \Sigma$. If $u_1$ and $u_2$ extend continuously to and agree on $\partial \Sigma$, then the open subset $U$ from the proof above does not intersect $\partial \Sigma$, and from this observation the proof goes through. This slight extension will be used in the proof of Theorem \ref{principal_D}.
\end{remark}
\end{subsection}

\begin{subsection}{Proof of Proposition \ref{prop: branched immersions}}\label{subsec: proof of theorem a}
   We now take a slight digression to prove Proposition \ref{prop: branched immersions}. We include this proof for the sake of completeness, and because we will reference it in the proof of Lemma \ref{lem: modified BBDH}. As in the statement of the proposition, let $f:\widetilde{X}\to (\mathbb{H}^2,\sigma)$ be an equivariant harmonic map with holomorphic data $(\phi,D)$, $\phi\neq 0$. 
   \begin{proof}[Proof of Proposition \ref{prop: branched immersions}]
For the "only if" direction, as in \cite[Lemma 3.2]{BBDH}, the sign of the Jacobian of a branched immersion does not flip. The divisor of the anti-holomorphic energy is $(\phi)-D$, so if $2D(p)<(\phi)(p)$, we have $J(f)>0$ near $p$, and if $2D(p)>(\phi)(p)$, we have $J(f)<0$ near $p$. Thus, for a branched immersion, $2D-(\phi)$ does not flip sign.
   
       For the main "if" direction, if $2D<(\phi)$, we apply Proposition \ref{prop: energy domination} with $u_1=\log L(f)$ and $u_2=\log H(f)$. Then, $D_1=(\phi)-D$ and $D_2=D$, and hence our assumption shows that $J(f)=H(f)-L(f)>0$ on the complement of the support of $D$. By the argument from \cite[pp. 12]{BBDH} (or using the Hartman-Wintner formula as in \cite{wood_paper}), the isolated singular points of $f$ are branch points. If $2D>(\phi)$, we apply Proposition \ref{prop: energy domination} with $u_1=\log H(f)$ and $u_2=\log L(f)$ and the argument is symmetric.
   \end{proof}
\end{subsection}

\subsection{Necessary and sufficient condition}\label{sec: nec and suf condition}
We now move on to the domination problem. Proposition \ref{prop: H and L conditions} below clarifies the role of $H$ and $L$. Define $$U_f =\{z\in X: H(f)(z)\geq L(f)(z)\}, \hspace{1mm} V_f = \{z\in X: H(f)(z)\leq L(f)(z)\}$$ and $$U_h =\{z\in X: H(h)(z)\geq L(h)(z)\}\hspace{1mm} V_h=\{z\in X: H(h)(z)\leq L(h)(z)\}.$$ Note that $U_f\cap V_f$ is the singular set of $f,$ and $U_h\cap V_h$ is the singular set of $h$.

\begin{prop}\label{prop: H and L conditions}
     Assume that $h$ and $f$ have the same Hopf differential. For $h^*\sigma\geq f^*\sigma$ to hold everywhere, it is necessary and sufficient that the following four conditions are satisfied. 
    \begin{enumerate}
        \item On $U_f\cap U_h,$ $H(h)\geq H(f)$.
        \item On $U_f\cap V_h,$ $L(h)\geq H(f)$.
        \item On $V_f\cap U_h,$ $H(h)\geq L(f)$.
        \item On $V_f\cap V_h$, $L(h)\geq L(f)$.
    \end{enumerate}
    For $h^*\sigma> f^*\sigma$ in a neighbourhood of a point $x$, it is necessary and sufficient that (1) if $x\in U_f\cap U_h$, $H(h)(x)>H(f)(x),$ (2) if $x\in U_f\cap V_h,$ $L(h)(x)>H(f)(x)$, and similar for (3) and (4).
\end{prop}
Here, $f^*\sigma\leq h^*\sigma$ (as opposed to $f^*\sigma<h^*\sigma$) means that $f^*\sigma(v,v)\leq h^*\sigma(v,v)$ for every unit tangent vector $v$.
\begin{proof}
From the formula (\ref{eq: pullback metric2}), $h^*\sigma\geq f^*\sigma$ if and only if $e(h)\geq e(f)$.
Rewriting $e=H+L$ as 
\begin{equation}\label{eq: expression for e}
    e=(H^{1/2}-L^{1/2})^2+2H^{1/2}L^{1/2}
\end{equation}
 and recalling the formula $|\phi|_\nu^2=H^{1/2}L^{1/2},$ the condition $e(h)\geq e(f)$ is equivalent to demanding that 
 \begin{equation}\label{eq: H^1/2-L^1/2}
     |H^{1/2}(h)-L^{1/2}(h)|\geq |H^{1/2}(f)-L^{1/2}(f)|.
 \end{equation}
 To make the notation easier on the eyes, note that \eqref{eq: H^1/2-L^1/2} is equivalent to 
$$|H(h)-L(h)|\geq |H(f)-L(f)|.$$ With this in mind, we check that $|H(h)-L(h)|\geq |H(f)-L(f)|$ is necessary and sufficient. We explicitly write out the proof only for $U_f\cap U_h$ and $U_f\cap V_h$ and leave the rest to the reader, since the arguments for $V_f\cap U_h$ and $V_f\cap V_h$ are totally analogous and don't add anything new.

On $U_f\cap U_h,$ assuming (1), we have $H(h)\geq H(f)\geq L(f)$. From $H(h)L(h)=H(f)L(f),$ we get  $L(f)\geq L(h).$ Hence $H(h)-L(h)\geq H(f)-L(f)\geq 0$, and taking absolute values yields the result. If (1) fails then we reverse the inequalities to see that $H(f)-L(f)\geq H(h)-L(h)\geq 0.$

On $U_f\cap V_h,$ assuming (2), $L(h)\geq H(f)\geq L(f).$ Similar to above, $H(h)L(h)=H(f)L(f)$ implies that $L(f)\geq H(h),$ and hence $L(h)-H(h)\geq H(f)-L(f)\geq 0.$ If (2) does not hold, we get $H(f)-L(f)\geq L(h)-H(h)\geq 0.$ 

As we said above, we omit the arguments for $V_f\cap U_h$ and $V_f\cap V_h$. The strictness statement is obtained by going back into the proof above and making the inequalities strict.
\end{proof}

\subsection{Theorems \ref{principal_A} and \ref{principal_C}}\label{sec: thm a and corollary a}
Throughout this subsection, we use the sets of the form $U_f$, $U_h$, $V_f,$ and $V_h$ from Section \ref{sec: nec and suf condition}. The main lemma is an application of Proposition \ref{prop: energy domination}.
\begin{lemma}\label{lem: H(f) less than H(h)}
    Let $f,h:\widetilde{X}\to (\mathbb{H}^2,\sigma)$ be equivariant harmonic maps with holomorphic data $(\phi,D_1)$ and $(\phi,D_2)$ respectively. If $D_2<D_1$, then $H(f)\leq H(h)$, strictly away from the support of $D_2$. If $D_2<(\phi)-D_1$, then $L(f)\leq H(h)$, strictly away from the support of $D_2$.
\end{lemma}
\begin{proof}
    For the first statement, set $u_1=\log H(f)$ and $u_2=\log H(h)$, which are Bochner solutions with divisors $D_1$ and $D_2$ matching the divisors from the statement of the lemma. The result is immediate from Proposition \ref{prop: energy domination}. For the second statement, we go through the same results, but take $u_1=\log L(f)$, which, as a Bochner solution, has divisor $(\phi)-D.$
\end{proof}

We are now ready to prove Theorems \ref{principal_A} and \ref{principal_C}. We first prove Theorem \ref{principal_C}, then Corollary \ref{corc}, from which we deduce Theorem \ref{principal_A}. 
\begin{proof}[Proof of Theorem \ref{principal_C}]
    By our assumptions, $U_h=X$. Assume that $f^*\sigma< h^*\sigma$. This rules out $D_1=D_2$ (for then we would have $f^*\sigma=h^*\sigma$ everywhere). The equality $D_1+D_2=(\phi)$ is ruled out too, since $\deg D_1\leq 2g-2$, and $h$ being a branched immersion implies that $\deg D_2<2g-2$. If $D_1(p)<D_2(p)$ at a point $p\in X$, then 
    \begin{equation}\label{eq: domination fails}
        \lim_{z\to p}\frac{H(h)}{H(f)}(z)=0.
    \end{equation}
As well, $D_1(p)<D_2(p)<(\phi)-D_2(p)<(\phi)-D_1(p)$, which implies that $p\in U_f=U_f\cap U_h$. Then, Proposition \ref{prop: H and L conditions}, specifically situation (1), contradicts $f^*\sigma\leq h^*\sigma$. Using similar reasoning, we will show that $D_1+D_2<(\phi)$. Suppose that $(D_1+D_2)(p)>(\phi)(p)$ at a point $p\in X$. Then, $$D_1(p)\geq D_2(p)>((\phi)-D_1)(p),$$ which implies $p\in V_f=U_h\cap V_f$. Analogous to (\ref{eq: domination fails}), $$\lim_{z\to p} \frac{H(h)}{L(f)}(z)=0,$$ which via Proposition \ref{prop: H and L conditions} implies that $f^*\sigma> h^*\sigma$ near $p$, and thus gives a contradiction. We conclude that $D_2<D_1$ and $D_1+D_2<(\phi).$

Conversely, assume that $D_2<D_1$ and that $D_1+D_2<(\phi).$ We write $X=U_f\cup V_f$. The condition $D_2<D_1$, together with Lemma \ref{lem: H(f) less than H(h)} and Proposition \ref{prop: H and L conditions}, imply that $f^*\sigma < h^*\sigma$ on $U_f$. The condition $D_1+D_2<(\phi)$ is equivalent to $D_2 < (\phi)-D_1$, which via Lemma \ref{lem: H(f) less than H(h)} and Proposition \ref{prop: H and L conditions} implies that $f^*\sigma<h^*\sigma$ on $V_f$. Hence, $f^*\sigma<h^*\sigma$ everywhere.
\end{proof}

\begin{proof}[Proof of Corollary \ref{corc}]
For both implications, we can assume that $h$ is a branched immersion, since this is implied by both $D_2<D_1$ (from Proposition \ref{prop: branched immersions}) and $f^*\sigma<h^*\sigma$. Since $\deg(D_1), \deg(D_2) \leq 2g - 2$, Proposition \ref{prop: branched immersions} implies $2D_1<(\phi)$ and $2D_2<(\phi)$. Hence, the condition $D_1+D_2<(\phi)$ is automatic. The corollary is then immediate from Theorem \ref{principal_C}.
\end{proof}

\begin{proof}[Proof of Theorem \ref{principal_A}]
Let $\rho$, $f$, and $k$ be as in the statement of the theorem and let $(\phi,D_1)$ be the holomorphic data of $f$. If $\phi\neq 0$, we pick any divisor $D_2$ of degree $\deg  D_1-k$ such that $D_2\leq D_1$, and we apply Proposition \ref{prop: holomorphic data} and Corollary \ref{corc} part (1) to produce the desired representation $j$ and equivariant harmonic map $h$. If $\phi=0$, we pick $D_2$ of degree $\deg  D_1-k$ and dominated by both $D_1$ and the divisor of a non-zero quadratic differential, and we apply Proposition \ref{prop: holomorphic data} and Corollary \ref{corc} as above.
\end{proof}
We conclude this subsection by showing that an assumption such as "$h$ is a branched harmonic immersion" is necessary. This result shows that the domination problem can become delicate.
\begin{prop}\label{prop: counterexample}
    For $\phi\neq 0$, consider equivariant harmonic maps $h$ and $f$ with holomorphic data $(\phi,D_1)$ and $(\phi,D_2)$ respectively. Assume that $D_2<D_1$. If there exist distinct points $p$ and $q$ at which $2D_1(p)<(\phi)(p)$ and $2D_2(q)>(\phi)(q)$ respectively, then neither $f^*\sigma\leq h^*\sigma$ nor $h^*\sigma \leq f^*\sigma$ hold everywhere.
\end{prop}
\begin{proof}
By Lemma \ref{lem: H(f) less than H(h)}, $H(f)\leq H(h)$ everywhere, strictly away from the support of $D_2$. Since $H(h)L(h)=H(f)L(f)$, we obtain $$J(h)=H(h)-L(h)\geq H(f)-L(f)=J(f)$$ everywhere, with the same strictness condition. We deduce that 
    \begin{equation}\label{eq: implications}
        J(h)\leq 0 \Rightarrow J(f)\leq 0, \hspace{1mm} J(f)\geq 0 \Rightarrow J(h)\geq 0.
    \end{equation}
By assumption, the open subset $U_{f}$ is non-empty (it contains the point $p$). Thus, by \eqref{eq: implications}, $U_{f}\subset U_{h}$, and by Proposition \ref{prop: H and L conditions} case (1), $h^*\sigma \leq f^*\sigma$ fails around $p$. On the other hand, the presence of the point $q$ shows that $V_{h}$ is non-empty, and \eqref{eq: implications} shows that $V_{h}\subset V_{f}$. By Proposition \ref{prop: H and L conditions} case (4), $h^*\sigma \leq f^*\sigma$ does not hold around $q$.  
\end{proof}

\subsection{Theorem \ref{principal_D}}
Both types of examples from Theorem \ref{principal_D} come from variations on the same construction. Let $X_1$ be a closed Riemann surface with one boundary component $\gamma$ and form the Riemann surface double $Y$, which comes with an antiholomorphic involution $\tau$ whose set of fixed points is $\gamma$. Necessarily, $Y$ has even genus, and of course any even genus smooth surface can be seen to arise from such a doubling. We attach a conformal metric to $Y$ so that we can define functions such as $H(\cdot)$ and $L(\cdot)$. Let $\phi$ be a non-zero holomorphic quadratic differential with no zeros on $\gamma$ and that is symmetric with respect to $\tau$, i.e, such that $\tau^*\phi =\overline{\phi}$. We view $X_1$ and $X_2:=\tau(X_1)$ as subsets of $Y$. We choose a first divisor $D'$ on $X_1$ dominated by $(\phi)|_{X_1}$. If $D''=(\phi)|_{X_2}$, we define $D=D'+(D''-\tau(D'))$. We point out that for such pairs, $\deg  D = \frac{\deg (\phi)}{2}=2g-2$, and hence the representation corresponding to $(\phi,D)$ has Euler number $0$ (recall Proposition \ref{prop: euler number}). Let $\widetilde{\tau}$ be the lift of $\tau$ to the universal cover $\widetilde{Y}$. 

For this $Y$ and $(\phi,D)$ as above, let $\rho:\Gamma\to \psl$ be the corresponding representation and let $f:\widetilde{Y}\to (\mathbb{H}^2,\sigma)$ be a $\rho$-equivariant harmonic map. 
\begin{lemma}\label{lem: symmetry}
     $f\circ \widetilde{\tau}=f$ and $J(f)=0$ on $\gamma.$
\end{lemma}
\begin{proof}
In a local coordinate $z$, since $\tau$ is anti-holomorphic, $(f\circ \widetilde{\tau})_z = (\overline{\widetilde{\tau}})_zf_{\overline{z}}(\tau (z)).$ It follows using the definitions stemming from (\ref{eq: pullback metric2}) that $\phi(f\circ \widetilde{\tau}) = \widetilde{\tau}^*\overline{\phi}(f)=\phi,$ and $H(f\circ \widetilde{\tau})=L(f)\circ \tau.$ The second equality implies that the divisor associated with $f\circ \widetilde{\tau}$ is $(\phi)-D\circ \tau =D$. Thus, $f$ and $f\circ \widetilde{\tau}$ have the same holomorphic data. By Proposition \ref{prop: holomorphic data}, if the representation associated with $(\phi, D)$ is irreducible, then $f=f\circ \widetilde{\tau}$, and if the representation is reductive but not irreducible, then $f$ and $f\circ \widetilde{\tau}$ are related by translation along a geodesic. Even in the latter case, since $\tau$ fixes $\gamma$, we can conclude that  the two maps agree. The fact that $J(f)=0$ on $\gamma$ follows from the equality $H(f\circ \widetilde{\tau})=L(f)\circ \tau.$
\end{proof}
Let $\widetilde{X}_1\subset \widetilde{Y}$ be the preimage of $X_1$ under the universal covering map. We record an analog of Proposition \ref{prop: branched immersions} for maps restricted to $\widetilde{X}_1$.
\begin{lemma}\label{lem: modified BBDH}
$f|_{\textrm{int}(\widetilde{X}_1)}$ is a branched immersion if and only if $2D'<(\phi|_{X_1})$ or $2D'>(\phi|_{X_1})$.
\end{lemma}
\begin{proof}
The same argument from the proof of Proposition \ref{prop: branched immersions} implies that $2D'<(\phi|_{X_1})$ or $2D'>(\phi|_{X_1})$ are necessary. For the converse, as in the proof of Proposition \ref{prop: branched immersions}, we apply Proposition \ref{prop: energy domination} and Remark \ref{rem: boundary case} with $u_1=\log L(f)|_{X_1}$ and $u_2=\log H(f)|_{X_1}$ if  $2D'<(\phi|_{X_1})$, and $u_1=\log H(f)|_{X_1}$ and $u_2=\log L(f)|_{X_1}$ if $2D'>(\phi|_{X_1})$. By Lemma \ref{lem: symmetry}, $u_1=u_2$ extend to and agree on $\partial X_1=\gamma$, so Remark \ref{rem: boundary case} indeed applies.
\end{proof}
Similar to above, we provide an analog for Corollary \ref{corc}. Let $(\phi,D_1)$ and $(\phi, D_2)$ be constructed as above with equivariant harmonic maps $f$ and $h$ respectively. For $i=1,2$, we let $D_i'$ and $D_i''$ be the divisors corresponding to $D'$ and $D''$ respectively.
\begin{lemma}\label{lem: thm a modified with boundary}
     Assume that $2D_1'<(\phi|_{X_1})$. Then $f^*\sigma< h^*\sigma$ if and only if $D_2'<D_1'$ or $D_2'>(\phi|_{X_1})-D_1'$.
\end{lemma}
\begin{proof}
By using an outer automorphism of the fundamental group, we can assume that $\deg D_2|_{X_1}\leq g-1$, and under this assumption the case $D_2'>(\phi|_{X_1})-D_1$ is removed. By Lemma \ref{lem: symmetry},  $f^*\sigma|_{X_2}=\tau^* (f^*\sigma|_{X_1})$ and $\tau^* (h^*\sigma|_{X_1})=h^*\sigma|_{X_2}$, so $f^*\sigma\leq h^*\sigma$ on $X_1$ if and only if $f^*\sigma\leq h^*\sigma$ everywhere. By Lemma \ref{lem: modified BBDH}, $f|_{\textrm{int}(\widetilde{X}_1)}$ is a branched immersion. 

The exact same local analysis from the proof of Theorem \ref{principal_C} shows that $D_2'<D_1'$ is necessary for the domination. Assume that $D_2'<D_1'$. By Lemma \ref{lem: symmetry}, both $J(f)$ and $J(h)$ vanish on $\gamma$, and from $H(f)L(f)=H(h)L(h),$ we obtain that $H(f)|_{\gamma}=H(h)|_{\gamma}.$ Setting $u_1=\log H(f)|_{X_1}$ and $u_2=\log H(h)|_{X_1}$, we apply Proposition \ref{prop: energy domination} and Remark \ref{rem: boundary case} to obtain $H(f)\leq H(h)$ on $X_1$ (analogous to Lemma \ref{lem: H(f) less than H(h)}), and the inequality is strict when $J(h)\neq 0$. By the argument from Theorem \ref{principal_C}, $f^*\sigma<h^*\sigma$ on $X_1$. By symmetry, the same is true on $X_2$. 
\end{proof}
With preparations complete, we now prove Theorem \ref{principal_D}.
\begin{proof}[Proof of Theorem \ref{principal_D}]
We work with the construction and notations as above. For (1), select pairs $(\rho,f)$ and $(j,h)$ giving rise to data $(\phi,D_2)$ and $(\phi, D_1)$ respectively with $2D_1'<(\phi|_{X_1})$ and $D_2'<D_1'$, exactly as in Lemma \ref{lem: thm a modified with boundary}. 

For (2), we consider $(\phi,D_0)$ as above and, for $D_0=D_0'+(D_0''-\tau(D_0'))$ as above, we specify that $D_0'=0$. Let $(\rho,f)$ be the corresponding representation and equivariant harmonic map. By Lemmas \ref{lem: symmetry} and \ref{lem: modified BBDH}, $f|_{\textrm{int}(\widetilde{X}_1)}$ and $f|_{\textrm{int}(\widetilde{X}_2)}$ are branched immersions, and $f$ is singular on the shared frontier of $\widetilde{X}_1$ and $\widetilde{X}_2$. In fact, since $D_0'=0$ (and $D_0''=(\phi)|_{X_2})$,  $f|_{\textrm{int}(\widetilde{X}_1)}$ and $f|_{\textrm{int}(\widetilde{X}_2)}$ are immersions. 

Let $(j,h)$ be any other pair with holomorphic data $(\phi,D)$ and assume that $f^*\sigma< h^*\sigma$. Since $f^*\sigma$ is non-degenerate on $\textrm{int}(X_1)$ and $\textrm{int}(X_2)$, the singular set of $h$ is contained in $\gamma$, and $h|_{\textrm{int}(\widetilde{X}_1)}$ and $h|_{\textrm{int}(\widetilde{X}_2)}$ are immersions. By applying an outer automorphism to $j$ (which will not alter $h^*\sigma$), we can assume that $J(h)|_{X_1}\geq 0$. Then $D|_{X_1}$ is contained in the singular set of $h|_{\textrm{int}(\widetilde{X}_1)}$, and the latter being empty forces $D|_{X_1}=0$. Now, we have that either $J(h)|_{X_2}\geq 0$ or $J(h)|_{X_2}\leq 0$. If $J(h)|_{X_2}\geq 0$, the non-degeneracy of $h^*\sigma$ on $\textrm{int}(X_2)$ shows that $D|_{X_2}=0$, which forces $j$ to be Fuchsian. Similarly, if $J(h) \leq 0$, then the divisor of the square root of the anti-holomorphic energy on $X_2$, which is by definition $(\phi|_{X_2}) - D|_{X_2}$, is contained in the singular set of $h|_{\mathrm{int}(\widetilde{X}_2)}$, and thus $(\phi|_{X_2}) = D|_{X_2}$. We therefore deduce by Proposition \ref{prop: holomorphic data} that $\rho$ and $j$ are conjugate. But then $f^*\sigma=h^*\sigma$, and we have a contradiction.
\end{proof}

\section{\texorpdfstring{Anti-de Sitter $3$-manifolds}{Anti-de Sitter 3-manifolds}}\label{sec4}
In this section, we introduce the geometry of the anti-de Sitter space and the local model of a spin-cone singular anti-de Sitter manifold. 
\subsection{Anti-de Sitter space}
Let $\mathfrak{sl}(2,\mathbb{R})$ be the Lie algebra of $\psl$. On $\mathfrak{sl}(2,\mathbb{R})$, we consider the bilinear form $-\mathrm{det}$, which is invariant under the adjoint representation. It can be shown that $-\mathrm{det}$ coincides with $\frac{1}{8}$ times the Killing form of $\psl$. This form has signature $(2,1)$ and induces a bi-invariant Lorentzian metric on $\psl$, which we denote by $g^{\ads}$. We define the $3$-dimensional anti-de Sitter space to be:
$$\ads := \left(\psl, g^{\ads}\right).$$
The group of orientation- and time-orientation–preserving isometries is $\psl \times \psl$, acting on $\ads$ by left and right multiplication. That is, 
$$(A, B) \cdot X = A X B^{-1}.$$
A tangent vector $v \in \mathrm{T}_x\ads$ is timelike if $g^{\ads}(v, v) < 0$, lightlike if $g^{\ads}(v, v) = 0$, and spacelike if $g^{\ads}(v, v) > 0$.
We call a geodesic $\gamma$ timelike if every tangent vector is timelike, lightlike if every tangent vector is lightlike, and spacelike if every tangent vector is spacelike.

It turns out that every timelike geodesic is of the form
\begin{equation}\label{eq_timelike}
\ell_{p,q} = \left\{ A \in \psl \mid A \cdot q = p \right\},
\end{equation}
for $(p, q) \in \mathbb{H}^2 \times \mathbb{H}^2$. These are topological circles and have Lorentzian length $\pi$. Note that under this identification, it can be checked that for any isometry $(A, B)$ of $\ads$, we have
$$(A, B) \cdot \ell_{p,q} = \ell_{Ap, Bq}.$$
Hence, the 1-to-1 correspondence $(p, q) \to \ell_{p,q}$ is equivariant with respect to the action of $\psl \times \psl$ on $\mathbb{H}^2 \times \mathbb{H}^2$ and on the set of timelike geodesics.\\

We end this subsection by briefly recalling the notion of a geometric structure on a $3$-manifold $M$. Let $X$ be a manifold and $G$ a Lie group acting transitively on $X$ by analytic diffeomorphisms. Then a $(G,X)$-structure on $M$ is a maximal atlas of coordinate charts on $M$ with values in $X$ such that the transition maps are given by elements of $G$. An important result from the theory is that $M$ is equipped with a \emph{holonomy representation} $\rho:\pi_1(M)\to G$ and a $\rho$-equivariant local diffeomorphism $ \mathrm{dev} : \widetilde{M} \to X$, called the \emph{developing map}. The pair $(\mathrm{hol},\mathrm{dev})$ is defined up to the action of $G$, where $G$ acts by conjugation on the holonomy representation and post-composition on the developing map. In this paper, we focus on the case where $X$ is the three-dimensional anti-de Sitter space $\ads$ and $G = \isom(\ads)$. The corresponding $(G,X)$-structures on $M$ are known as \emph{anti-de Sitter structures}, and a manifold endowed with such a structure is called an \emph{anti-de Sitter manifold}. For further details on three-dimensional anti-de Sitter geometry, we refer the reader to~\cite{BS_ads}.

\subsection{AdS manifolds with spin-cone structure}
Before getting to spin-cone singularities, we recall the model of cone singularities in $2$-dimensional hyperbolic geometry.
\subsubsection{Hyperbolic cone singularities}
Let $\mathbb{H}^2_* := \mathbb{H}^2 \setminus \{i\}$, and let $c: [0,+\infty) \to \mathbb{H}^2$ be the geodesic parametrized by hyperbolic arc length given by
$$
c(t) = i e^t,
$$
so that $c(0) = i$ and $\lim_{t \to +\infty} c(t) = \infty$ in the boundary at infinity of $\mathbb{H}^2$. The universal cover $\widetilde{\mathbb{H}^2_*}$ of $\hh$ is given by
\begin{equation}\label{eq_pi_map}
\begin{array}{rcl}
\pi & : & (0, +\infty) \times \mathbb{R} \to \mathbb{H}^2_* \\
    &   & (r, \theta) \mapsto R^{\theta}c(r).
\end{array}
\end{equation}
The group of deck transformations of $\pi$, which is isomorphic to the fundamental group of $\mathbb{H}^2_*$, is given by
\[
\pi_1(\mathbb{H}^2_*) = \left\{ (r, \theta) \mapsto (r, \theta + 2k\pi) \mid k \in \mathbb{Z} \right\}.
\]
We endow $\widetilde{\mathbb{H}^2_*}$ with the Riemannian metric obtained by pulling back the hyperbolic metric $\sigma$ via $\pi$. In this model, the isometry group satisfies $\mathrm{Isom}(\widetilde{\mathbb{H}^2_*}) \cong \mathbb{R}$. More precisely,
\[
\mathrm{Isom}(\widetilde{\mathbb{H}^2_*}) = \left\{ (r, \theta) \mapsto (r, \theta + \tau) \mid \tau \in \mathbb{R} \right\}.
\]
For $\theta_0 \in \mathbb{R}$, we define the local model of a \emph{hyperbolic cone structure} by
\begin{equation}\label{eq_cone}
\mathbb{H}^2_{\theta_0} := \widetilde{\mathbb{H}^2_*}/(r, \theta) \sim (r, \theta + \theta_0).
\end{equation}
We call the point $i$ the \textit{branched point} of $\hh$. To justify this terminology, we define the \textit{universal branched cover} of $\mathbb{H}^2_*$ as the quotient:
\begin{equation}
\left( [0,+\infty) \times \mathbb{R} \right) \sqcup \{0\} := \left( [0,+\infty) \times \mathbb{R} \right) / \sim,
\end{equation}
where \((0, \theta) \sim (0, \theta')\) for all \(\theta, \theta' \in \mathbb{R}\), identifying all points at \(r = 0\) to a single point, denoted \(\{0\}\). The universal covering map \(\pi: (0, +\infty) \times \mathbb{R} \to \mathbb{H}^2_*\) extends to this space by setting \(\pi([0, \theta]) = i\) for all \(\theta\). This is well defined since \(c(0) = i\) and \(R^{\theta} i = i\). Thus, \(\pi(0) = i\), making $i$ the branch point of the covering $\pi$.

As a particular case, it is worth observing that when $\theta_0 = 2n\pi \in 2\pi\mathbb{Z}$, the surface $\mathbb{H}^2_{2n\pi}$ is a degree-$n$ cover of $\mathbb{H}^2_*$. Indeed, the universal covering map $\pi: \widetilde{\mathbb{H}^2_*} \to \mathbb{H}^2_*$ induces the degree-$n$ cover 
\[
\overline{\pi}: \mathbb{H}^2_{2n\pi} \to \mathbb{H}^2_*, \quad [r, \theta]_n \mapsto \pi(r, \theta),
\]
where $[r, \theta]_n$ denotes the class of $(r, \theta)$ in $\mathbb{H}^2_{2n\pi}$. We also note that $\mathbb{H}^2_{2n\pi}$ can be identified with $\mathbb{H}^2_*$ via the map
\begin{equation}\label{identification_punctured}
\begin{array}{ccccc}
 & & \mathbb{H}^2_{2n\pi} & \to & \mathbb{H}^2_* \\
 & & [R, \theta]_n & \mapsto & R^{\frac{\theta}{n}} c(r).
\end{array}
\end{equation}
Therefore, it will be useful to view $\mathbb{H}^2_{2n\pi}$ as the punctured hyperbolic plane $\mathbb{H}^2_*$ equipped with the degree-$n$ covering map
\[
\begin{array}{ccccc}
 & & \mathbb{H}^2_{2n\pi} \cong \mathbb{H}^2_* & \to & \mathbb{H}^2_* \\
 & & R^{\frac{\theta}{n}} c(r) & \mapsto & R^{\theta} c(r).
\end{array}
\]
In the disc model, this map corresponds to $z \mapsto z^n$.

\subsubsection{Spin-cone structure}
In what follows, for each \( r, \theta \in \mathbb{R} \), we consider the following isometries of the hyperbolic Poincar\'e half plane:
\[
A(r) =
\begin{pmatrix}
e^{\frac{r}{2}} & 0 \\
0 & e^{-\frac{r}{2}}
\end{pmatrix}, \quad
R^{\theta} =
\begin{pmatrix}
\cos\left(\frac{\theta}{2}\right) & \sin\left(\frac{\theta}{2}\right) \\
-\sin\left(\frac{\theta}{2}\right) & \cos\left(\frac{\theta}{2}\right)
\end{pmatrix}.
\]
The matrix \( A(r) \) acts as a translation of length \( r \) along the geodesic in \( \mathbb{H}^2 \) with endpoints \( 0 \) and \( \infty \), whereas \( R^{\theta} \) acts as a rotation of angle \( \theta \) fixing the point \( i \in \mathbb{H}^2 \) (alternatively, one may view it as a rotation fixing the origin in the Poincar\'e disc model). Let \( \ell_{i,i} \) denote the set of elliptic isometries fixing \( i \in \mathbb{H}^2 \), see \eqref{eq_timelike}. Observe that the curve \( [0, \pi] \to \ads \), given by \( \theta \mapsto R^{2\theta} \), is the arc-length parametrization of \( \ell_{i,i} \). Hence, the Lorentzian length of \( \ell_{i,i} \) is equal to \( \pi \), that is,  
\begin{equation}\label{Length}
\int_0^{\pi} \sqrt{-\det((R^{2\theta})')} \, d\theta = \pi.
\end{equation}

Now, consider the space \( \ads_* := \mathrm{PSL}(2, \mathbb{R}) \setminus \ell_{i,i} \). In \cite[Proposition 3.5.1]{Agnese_thesis}, Janigro defines the universal cover of \( \ads_* \) by the map  
\begin{equation}\label{agnesmap}
\begin{array}{cccc}
T & : & (0, +\infty) \times \mathbb{R} \times \mathbb{R} & \to \ads_* \\
  &   & (r, \theta, t) & \mapsto R^{\theta} A(r) R^{-t}.
\end{array}
\end{equation}
The group of deck transformations of $T$, which is isomorphic to the fundamental group of $\ads_*$, is given by

\[
\pi_1(\ads_*) = \{ (r, \theta, \eta) \mapsto (r, \theta + 2k_1\pi, \eta + 2k_2\pi) \mid k_1, k_2 \in \mathbb{Z} \}.
\]
We endow \( \widetilde{\ads_*} \) with the Lorentzian metric obtained by pulling back the anti-de Sitter metric \( g^{\ads} \) via \( T \). In this model, \( \mathrm{Isom}(\widetilde{\ads_*}) \cong \mathbb{R}^2 \). More precisely, we have
\begin{equation}\label{isometry_ads}
\mathrm{Isom}(\widetilde{\ads_*}) = \{\varphi_{(\theta_0,\eta_0)}: (r, \theta, \eta) \mapsto (r, \theta + \theta_0, \eta + \eta_0) \mid \theta_0, \eta_0 \in \mathbb{R} \}.
\end{equation}
Note that the isometry group of \( \ads_* \) identifies with isometries of \( \ads \) that fix the timelike geodesic \( \ell_{i,i} \), that is,
\[
\mathrm{Isom}(\ads_*) = \mathrm{PSO}(2) \times \mathrm{PSO}(2) < \mathrm{PSL}(2, \mathbb{R}) \times \mathrm{PSL}(2, \mathbb{R}).
\]
The map \( T \) induces a homomorphism \( T_* : \mathrm{Isom}(\widetilde{\ads_*}) \to \mathrm{Isom}(\ads_*) \). If \( \varphi \in \mathrm{Isom}(\widetilde{\ads_*}) \), then \( T_*(\varphi) \) satisfies:
\begin{equation}\label{eq_T_morphism}
T_*(\varphi) \circ T = T \circ \varphi.
\end{equation}
The kernel of \( T_* \) is the group of deck transformations of the covering \( T \), given by
\begin{equation}\label{eq_automorphism_T}
\mathrm{Ker}(T_*) = \{ (r, \theta, \eta) \mapsto (r, \theta, \eta) + k_1(0, 0, 2\pi) + k_2(0, 2\pi, 0) \mid k_1, k_2 \in \mathbb{Z} \}.
\end{equation}
\begin{remark}\label{remark_on_isometries}
If we take $\ell$ to be another timelike geodesic in $\ads$, then there exists an isometry of $\ads$ sending $\ell_{i,i}$ to $\ell$. Hence, there is an isometry between $\widetilde{\ads_*}$ and $\widetilde{\ads \setminus \ell}$, and we still denote the isometries of $\widetilde{\ads \setminus \ell}$ by $\varphi_{(\theta_0,\eta_0)}$ as in \eqref{isometry_ads}.
\end{remark}
For \( \theta_0, \eta_0 \in \mathbb{R} \), we define \( \Lambda(\theta_0, \eta_0) \) as the lattice in \( \mathrm{Isom}(\widetilde{\ads_*}) \) generated by \( (0, \theta_0, \eta_0) \) and $(0,0,2\pi)$. That is,
\begin{equation}\label{lattice_ads}
\Lambda(\theta_0, \eta_0) = \{ (r, \theta, \eta) \mapsto (r, \theta, \eta) + k_1(0, 0, 2\pi) + k_2(0, \theta_0, \eta_0) \mid k_1, k_2 \in \mathbb{Z} \}.
\end{equation}
Observe that \( \Lambda(\theta_0, \eta_0) = \Lambda(\theta_0, \eta_0 + 2k\pi) \) for any \( k \in \mathbb{Z} \). Hence, \( \eta_0 \) may be regarded as an element of \( \mathbb{R} / 2\pi\mathbb{Z} \). Without loss of generality, we may assume \( 0 \leq \eta_0 < 2\pi \). Following \cite{Agnese_thesis}, we give the definition below.

\begin{defi}
    Let \( \theta_0 \in \mathbb{R} \) and \( 0 \leq \eta_0 < 2\pi \). We define the local model of a spin-cone singularity as  
    \[
    \ads_{(\theta_0, \eta_0)} := \widetilde{\ads_*} / \Lambda(\theta_0, \eta_0).
    \]
\end{defi}
The next lemma can be viewed as a generalization, in the singular setting, of the fact that the anti-de Sitter space \( \ads \) is a circle bundle over \( \mathbb{H}^2 \), with fibers being timelike geodesics of length \( \pi \).

\begin{lemma}
    Let \( \theta_0 \in \mathbb{R} \) be different from zero, and \( 0 \leq \eta_0 < 2\pi \). Consider the projection map \( (r,\theta,\eta) \mapsto (r,\theta) \), which induces a map \( \mathcal{F}: \ads_{(\theta_0,\eta_0)} \to \mathbb{H}^2_{\theta_0} \). Then, $\mathcal{F}$ is a fibration with the property that each fiber is a timelike geodesic of length \( \pi \).
\end{lemma}

\begin{proof}
Consider the projection map \( (r,\theta,\eta) \mapsto (r,\theta) \), which induces a map \( \mathcal{F}: \ads_{(\theta_0,\eta_0)} \to \mathbb{H}^2_{\theta_0} \). The fiber above \( [r,\theta] \) is given by
\[
\mathcal{F}^{-1}([r,\theta]) = \{[r,\theta,\eta] \mid \eta \in \mathbb{R} \}.
\]
We claim that \( \mathcal{F}^{-1}([r,\theta]) \) is a timelike geodesic with arc-length parametrization
\[
c(t) = [r,\theta,2t], \quad t \in [0,\pi).
\]
Let \( t_1, t_2 \in [0,\pi) \) be such that \( [r,\theta,2t_1] = [r,\theta,2t_2] \). Then there exist \( n, m \in \mathbb{Z} \) such that
\[
2t_1 = 2t_2 + 2n\pi + m\eta_0, \quad \text{and} \quad m\theta_0 = 0.
\]
Since \( \theta_0 \neq 0 \), it follows that \( m = 0 \) and hence \( 2t_1 = 2t_2 + 2n\pi \). Since \( t_1, t_2 \in [0,\pi) \), this implies \( n = 0 \) and \( t_1 = t_2 \), so \( c \) is injective.

Next, the curve
\[
\beta(t) = R^{2t} A(r_0)
\]
parametrizes a timelike geodesic of \( \ads \) by arc length. Hence, the fiber \( \mathcal{F}^{-1}([r,\theta]) \) is a timelike geodesic of length \( \pi \). This completes the proof.
\end{proof}
We now consider the special case \( \theta_0 = 2n\pi \) for some \( n \in \mathbb{N} \) and \( \eta_0 = 0 \), which will be our main focus. First of all, it can be shown without difficulties that the universal covering map $T : (0, +\infty) \times \mathbb{R}\times\mathbb{R} \to \ads_*$
induces a diffeomorphism $\ads_{(2\pi,0)}\to \ads_*$, which we continue to denote by $T$ (see \cite[Proposition 3.5.1]{Agnese_thesis}). For each integer $n$, we denote by $\lambda_n$ the map from $(0, +\infty) \times \mathbb{R}\times\mathbb{R}\to (0, +\infty) \times \mathbb{R}\times\mathbb{R}$ that sends $(r,\theta,\eta)$ to $(r,\frac{\theta}{n},\eta)$. According to the previous result, it is straightforward to check that $T\circ\lambda_n$ induces a diffeomorphism $\ads_{(2n\pi,0)}\to \ads_*$, for which we keep the notation $T\circ\lambda_n$. We denote by \( [r, \theta, \eta]_n \) the equivalence class of \( (r, \theta, \eta) \) under the action of \( \Lambda(2n\pi, 0) \). We observe that
\begin{equation}\label{mapC}
\begin{array}{cccc}
\mathcal{C}_n : & \ads_{(2n\pi,0)} & \to & \ads_{(2\pi,0)} \\
    & [r, \theta, \eta]_n & \mapsto & [r, \theta, \eta]_1,
\end{array}
\end{equation} is a degree-\( n \) covering map, which is also the case for the map $\mathcal{T}_n$ defined by
\begin{equation}
\begin{array}{cccc}
\mathcal{T}_n:=T\circ C_n\circ (T\circ\lambda_n)^{-1} : & \ads_{*} & \to & \ads_{*} \\
    & R^{\frac{\theta}{n}}A(r)R^{-\eta} & \mapsto & R^{\theta}A(r)R^{-\eta}.
\end{array}
\end{equation}
We summarize this discussion in the following lemma.
\begin{lemma}\label{lemma_covring_ads}
    The space \( \ads_{(2n\pi,0)} \) is a degree-\( n \) covering of \( \ads_* \).
\end{lemma}
We now move on to define anti-de Sitter manifolds with spin-cone singularities. Let \( M \) be an oriented three-manifold, and let \( L \) be a link in \( M \), i.e., a finite disjoint union of embedded circles \( K_i \subset M \). For each \( K_i \), we consider \( T_i \), a tubular neighborhood of \( K_i \). Each such neighborhood is homeomorphic to the solid torus \( \mathbb{D}^2 \times \mathbb{S}^1 \), and the complement \( T_i \setminus K_i \) is homeomorphic to \( \mathbb{D}^2_* \times \mathbb{S}^1 \), where \( \mathbb{D}^2_* := \mathbb{D}^2 \setminus \{(0,0)\} \) denotes the punctured disc.

On the universal cover of \( \mathbb{D}^2_* \times \mathbb{S}^1 \), we introduce cylindrical coordinates \( (r, x, y) \in (0,1) \times \mathbb{R} \times \mathbb{R} \) such that the universal covering map is given by
\begin{equation}\label{cylindrical}
\begin{array}{ccccc}
\mathcal{C} &  :& (0,1)\times\mathbb{R}\times\mathbb{R} & \to & \mathbb{D}^2_*\times\mathbb{S}^1 \\
 & & (r,x,y) & \mapsto & (re^{ix},e^{iy}).
\end{array}
\end{equation}
We define the \textit{universal branched cover of \( \mathbb{D}^2 \times \mathbb{S}^1 \) branched over \( \mathbb{S}^1 \)} as the quotient:\\
\begin{equation}
    \left( [0, 1) \times \mathbb{R} \times \mathbb{R} \right) \sqcup \mathbb{R} := \left( [0, 1) \times \mathbb{R} \times \mathbb{R} \right) / \sim,
\end{equation}
where \( (0, x, y) \sim (0, x', y) \) for any \( x, x' \in \mathbb{R} \) and \( y \in \mathbb{R} \). The real line \( \mathbb{R} \) attached to \( [0, 1) \times \mathbb{R} \times \mathbb{R} \) can thus be identified with \( \{ [0, 0, y] \mid y \in \mathbb{R} \} \), collapsing the \( \mathbb{R}^2 \)-plane along the \( x \)-axis.
Furthermore, observe that the covering map \( \mathcal{C} \) defined in \eqref{cylindrical} extends to \( \left( [0, 1) \times \mathbb{R} \times \mathbb{R} \right) / \sim \) by taking \( \mathcal{C}([0,x,y]) = (0,e^{iy}) \). Since \( T_i \setminus K_i \cong \mathbb{D}^2_* \times \mathbb{S}^1 \), we may similarly define the universal cover of \( T_i \) branched over \( K_i \) and denote it by \( \widetilde{T_i \setminus K_i} \sqcup \widetilde{K_i} \).
\begin{defi}\label{defi_spin_cone}
Let \( L \) be a link in \( M \) as defined above. We say that an \textit{anti-de Sitter structure} on \( M \setminus L \) has \textit{spin-cone singularities} along \( L \) if the following conditions hold:
\begin{itemize}
    \item The restriction of the developing map
    \[
    \mathrm{Dev} : \widetilde{T_j \setminus K} \to \ads
    \]
    extends continuously to \( \widetilde{T_j \setminus K_j} \sqcup \widetilde{K_j} \). Namely, in cylindrical coordinates \( (r, \theta, \eta) \in (0, 1) \times \mathbb{R} \times \mathbb{R} \), the limit
    \begin{equation}\label{extension_dev_branched}
        \lim_{r \rightarrow 0} \mathrm{Dev}(r, \theta, \eta) =: f(\eta)
    \end{equation}
    exists and is independent of \( \theta \). See Figure~\ref{WR_dev}.
    
    \item The map \( f \) sends \( \widetilde{K_j} \) onto a complete timelike geodesic \( \ell_j \subset \ads \).
    
    \item The lifted holonomy \( \rho : \pi_1(T_j \setminus K_j) \to \mathrm{Isom}(\widetilde{\ads\setminus\ell_j}) \) around a meridian \( \alpha \) encircling \( K_j \) is given by \( \varphi_{(\theta_j,\eta_j)} \) for some \( \theta_j, \eta_j \in \mathbb{R} \), and the holonomy of a longitude \( \beta \) is \( \varphi_{(0, 2\pi)} \) (see \eqref{isometry_ads} for notation).
\end{itemize}
We say that the structure is \emph{branched} (or that $M$ is a branched AdS manifold) if $\theta_j, \eta_j \in 2\pi \mathbb{Z}$.
\end{defi}

As a consequence of the above definition, each tubular neighborhood \( T_j \setminus K_j \) is locally isometric to the local model of a spin-cone singular manifold \( \ads_{(\theta_j,\eta_j)} \).

\begin{figure}[htb]
\centering
\includegraphics[width=.7\textwidth]{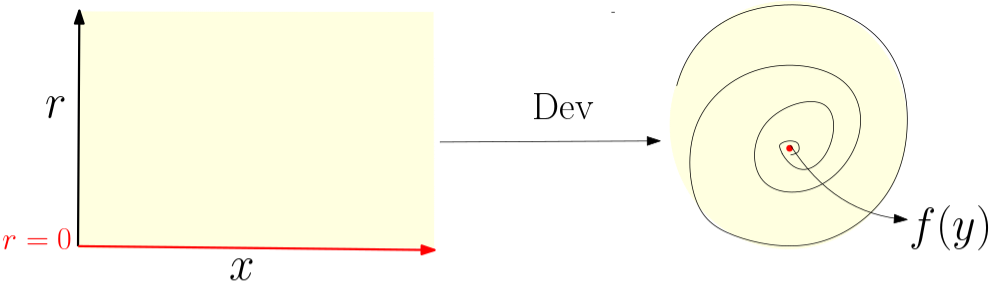}
\caption{The behavior of the developing map restricted to the slice \( (0, 1) \times \mathbb{R} \times \{y\} \).}
\label{WR_dev}
\end{figure}

\section{From domination to anti-de Sitter manifolds}\label{sec5}
The principal result of this section is Theorem \ref{ads_manifold}. Our construction builds on Janigro's thesis \cite{Agnese_thesis} but extends it in several essential ways. In her work, Janigro considers a pair of maps defined on the complement of the singular set of a hyperbolic cone surface, with $h$ an immersion. In contrast, we consider a general pair of dominating maps $f, h$ defined on closed hyperbolic surfaces, with no restriction on $h$ being an immersion. In particular, we provide a precise definition of the AdS manifold associated with such a pair of dominating maps and introduce its “completed” version that includes the singular locus (see Definition \ref{fhads_fhadss}). This refinement yields new topological results (Proposition \ref{topology_of_ads_fh}) beyond those established in \cite{Agnese_thesis}. These results play an important role in the proof of our main theorem on AdS manifolds (Theorem \ref{principal_ads}).

Throughout this section, we consider an oriented surface $\Sigma$ (not necessarily closed) with fundamental group $\Gamma$, and two smooth maps $f,h : \widetilde{\Sigma} \to \mathbb{H}^2$ such that
\begin{enumerate}[label=(\alph*)]
    \item \label{cond:h_singular} $h$ is a local diffeomorphism on an open dense subset. We denote by $C \subset \widetilde{\Sigma}$ the subset of points where $h$ is singular. Note that any harmonic map either has this property or has an image contained in a geodesic \cite[Theorem 3]{Sampson}.
    \item Denote by $\{U_\alpha\}_{\alpha\in I}$ an open cover of $\widetilde \Sigma\setminus C$
such that $h|_{U_\alpha}$ is a diffeomorphism onto its image.  
We require the cover to be $\Gamma$\nobreakdash-invariant, in the sense that the indices $I$ are labeled equivariantly under the deck transformation action
\begin{equation}\label{label_indice}\textit{}
\gamma \cdot \alpha = \beta \text{ if and only if } \gamma \cdot U_{\alpha} = U_{\beta}.
\end{equation}

    \item \label{cond:h_dominates} The map $h$ dominates $f$, i.e. $f^*\sigma < h^*\sigma$. In particular, we may shrink the neighborhood $U_{\alpha}$ so that for all $x,y \in U_{\alpha}$, we have
 $$
    d_{\sigma}\bigl(f(x),f(y)\bigr) < d_{\sigma}\bigl(h(x),h(y)\bigr).$$
        
\end{enumerate}

\subsection{Gluing of the fibration}
\label{sec:local}

For each open set $U_{\alpha}$, we define
\[
   M_\alpha 
   := \Bigl\{\,A \in \psl \;\Bigm|\; \exists!\, p \in U_\alpha
     \text{ such that } A(f(p)) = h(p)\Bigr\}.
\]
It turns out that $M_{\alpha}$ is an open subset of $\ads$, foliated by timelike geodesics, a fact observed in \cite[Proposition 7.2]{GK_lamination}.

\begin{prop}[{\cite[Proposition~6.1]{Almost_domination_Nat}}]
\label{prop:local_fibration}
For each $\alpha \in I$, the subset $M_{\alpha} \subset \ads$ is open.  
Moreover, the map
$$
  \mathcal F_\alpha : M_\alpha \longrightarrow U_\alpha,
  \quad
  A \longmapsto p,
$$
where $p$ is the unique point in $U_{\alpha}$ satisfying $A(f(p)) = h(p)$, defines a principal $\mathbb{S}^1$–bundle whose fibers are timelike geodesics in $\ads$.
\end{prop}
\medskip
Next, for each branched point $x \in C$, consider the timelike geodesic
\[
  \ell_x :=\ell_{f(x),h(x)} =\{\,A\in\psl\mid A(h(x))=f(x)\}\cong\mathbb{S}^1.
\]
For each $\alpha\in I$ and $x\in C$, we denote by $\mathcal{I}_{\alpha}:M_{\alpha}\to \ads$ and $\mathcal{I}_x:\ell_x\to \ads$ the inclusion maps.
We then assemble two disjoint unions:
\begin{equation}
\label{spaces_R_S}
\mathcal{R}_{f,h}
  \;=\;
\bigsqcup_{\alpha \in I} M_\alpha,
\qquad 
\mathcal{S}_{f,h}
  \;=\;
\bigsqcup_{x \in C} \ell_x,
\end{equation}
where the letters \(\mathcal{R}\) and \(\mathcal{S}\) stand for \emph{regular} and \emph{singular}, respectively.  Define two maps:
\begin{itemize}
  \item 
  \[
    \mathcal{I} \;:\; \mathcal{R}_{f,h} \sqcup \mathcal{S}_{f,h} 
    \;\longrightarrow\; \ads
    \quad \text{by}\quad
    \begin{cases}
      \mathcal{I}\bigl|_{M_\alpha} = \mathcal{I}_{\alpha}, \\
      \mathcal{I}\bigl|_{\ell_x} = \mathcal{I}_x,
    \end{cases}
  \]
  
  \item 
  \[
    \mathcal{F} \;:\; \mathcal{R}_{f,h} \sqcup \mathcal{S}_{f,h} 
    \;\longrightarrow\; \widetilde{\Sigma}
    \quad \text{by}\quad
    \begin{cases}
      \mathcal{F}\bigl|_{M_\alpha} = \mathcal{F}_\alpha, \\
      \mathcal{F}\bigl|_{\ell_x} = \text{(constant map equal to } x \text{).}
    \end{cases}
  \]

\end{itemize}

We equip \( \mathcal{R}_{f,h}\sqcup \mathcal{S}_{f,h} \) with the \textit{initial topology} induced by the map
\[
\mathcal{R}_{f,h}\sqcup \mathcal{S}_{f,h} 
\;\longrightarrow\; 
\ads \times \widetilde{\Sigma},
\qquad
A \;\longmapsto\; \bigl(\mathcal{I}(A),\,\mathcal{F}(A)\bigr).
\]
In particular, a sequence \(A_n\in \mathcal{R}_{f,h}\sqcup \mathcal{S}_{f,h}\) converges to \(A\in \mathcal{R}_{f,h}\sqcup \mathcal{S}_{f,h}  \) if and only if 
\[
\mathcal{I}(A_n)\;\longrightarrow\;\mathcal{I}(A) \quad\text{in }\ads
\quad\text{and}\quad
\mathcal{F}(A_n)\;\longrightarrow\;\mathcal{F}(A)\quad\text{in }\widetilde{\Sigma}.
\]

\begin{defi}\label{fhads_fhadss}
Let \( f, h: \widetilde{X} \to \mathbb{H}^2 \) be as above.  We define
\[
\fhads \;:=\; \bigl(\mathcal{R}_{f,h}\sqcup \mathcal{S}_{f,h}\bigr)\;\big/\!\sim,
\qquad
\fhadss \;:=\; \mathcal{R}_{f,h}\;\big/\!\sim,
\]
where $A,B$ in $\mathcal{R}_{f,h}\sqcup\mathcal{S}_{f,h}$ (or in $\mathcal{R}_{f,h}$ ) are equivalent if and only if
\begin{equation}\label{eq_relation}
\mathcal{I}(A) = \mathcal{I}(B)
\quad\text{and}\quad
\mathcal{F}(A) = \mathcal{F}(B).    
\end{equation}
\end{defi}
Using the equivalence relation above, we may define maps
\[
\mathcal{I} \;:\; \mathcal{R}_{f,h}\sqcup \mathcal{S}_{f,h} \;\longrightarrow\; \ads
\quad\text{and}\quad
\mathcal{F} \;:\; \mathcal{R}_{f,h}\sqcup \mathcal{S}_{f,h} \;\longrightarrow\; \widetilde{\Sigma},
\]
which descend to well‐defined maps on $\fhads$. We state now the principal result of this section.

\begin{theorem}\label{ads_manifold}
    Let $f,h:\widetilde{\Sigma}\to \mathbb{H}^2$ be maps satisfying conditions \ref{cond:h_singular}--\ref{cond:h_dominates}. Then,
    \begin{itemize}
        \item $\fhads$ is topologically a solid torus.
        \item $\fhadss$ is an anti–de Sitter manifold with an atlas of charts
       $$
          \bigl\{\,(\Pr(M_\alpha),\;\mathcal{I}\circ(\Pr|_{M_\alpha})^{-1})\bigr\}_{\alpha \in I}.
        $$
        Moreover, $\mathcal{F}:\fhadss\to\widetilde{\Sigma}\setminus C$ is a principal $\mathbb{S}^1$–bundle with timelike geodesic fibers.
    \end{itemize}
\end{theorem}
\begin{remark}
Although the constructions of $\fhads$ and $\fhadss$ rely on the choice of an open covering $\mathcal{U} = \{U_{\alpha}\}_{\alpha \in I}$ of $\widetilde{\Sigma} \setminus C$, different choices yield a canonical identification. To explain this, we temporarily denote the spaces by $\fhads(\mathcal{U})$ and $\fhadss(\mathcal{U})$ (instead of $\fhads$ and $\fhadss$) to keep track of the covering, and let $[A]_{\mathcal{U}}$ denote the equivalence class of $A \in \psl$ in $\fhads(\mathcal{U})$. Then, for another covering $\mathcal{V}$ of $\widetilde{\Sigma} \setminus C$, the natural map
$$\fhads(\mathcal{U}) \longrightarrow \fhads(\mathcal{V}), \qquad [A]_{\mathcal{U}} \longmapsto [A]_{\mathcal{V}}
$$ defines a homeomorphism between $\fhads(\mathcal{U})$ and $\fhads(\mathcal{V})$, which restricts to an isometry between the resulting AdS manifolds $\fhadss(\mathcal{U})$ and $\fhadss(\mathcal{V})$. We refer the reader to \cite[Section 3.3, p. 58]{Agnese_thesis} for a discussion of this functorial behavior in her (similar) context.
\end{remark}

\subsection{Topology of the gluing}
We start with the following proposition, which describes the topology of $\fhads$.
\begin{prop}\label{topology_of_ads_fh}
There is a  homeomorphism $\fhads\to h(\widetilde{\Sigma})\times\mathbb{S}^1$ that restricts to a homeomorphism between $\fhadss$ and $h(\widetilde{\Sigma}\setminus C)\times\mathbb{S}^1$. 
\end{prop}

To prepare the argument, consider the disjoint union
\begin{equation}\label{eq:union_of_charts}
\mathcal{X}
\;=\;
\bigsqcup_{\alpha \in I} U_\alpha
\;\;\bigsqcup\; C,
\end{equation}
where each \(U_\alpha\subset \widetilde{\Sigma}\setminus C\) is an open set as before. We now define a map
$
H \;:\; \mathcal{X} \;\longrightarrow\; \mathbb{H}^2
$
by
\begin{equation}
H(y):=h(y).
\end{equation}
for $y$ in $U_{\alpha}$ or in $C$. We endow \(\mathcal{X}\) with the \emph{initial topology} making \(H\) continuous; that is, a subset \(V \subset \mathcal{X}\) is open if and only if 
\[
V \;=\; H^{-1}(W)
\quad\text{for some open } W \subset h(\widetilde{\Sigma}).
\]
We then define on \(\mathcal{X}\) the equivalence relation
\[
x \sim y
\quad\Longleftrightarrow\quad
H(x) \;=\; H(y).
\]
\begin{lemma}\label{lem:gluing_of_X}
The quotient \(\mathcal{X}/\!\sim\) is homeomorphic to the image \(h(\widetilde{\Sigma}) \subset \mathbb{H}^2\). 
\end{lemma}
\begin{proof}
Let $q: \mathcal{X} \to\; \mathcal{X}/\!\sim$ be the quotient map, and define
\[
  \overline{H} \;:\; \mathcal{X}/\!\sim \;\longrightarrow\; h(\widetilde{\Sigma}),
  \quad
  \overline{H}\bigl(q(x)\bigr) \;=\; H(x).
\]
Since $H$ satisfies
\[
  H(x) = H(y)
  \quad\Longleftrightarrow\quad
  x \sim y,
\]
the induced map $\overline{H}$ is well-defined. Moreover, by the definition of the initial topology, this map is continuous. The inverse of $\overline{H}$ is given by
\[
  \overline{H}^{-1} \;:\; h(\widetilde{\Sigma}) \;\longrightarrow\; \mathcal{X}/\!\sim,
  \quad
  \overline{H}^{-1}(z) \;=\; q(p),
\]
where \( p \in \mathcal{X} \) is any point satisfying \( H(p) = z \).  
We claim that \( \overline{H}^{-1} \) is continuous. Let \( U' \subset \mathcal{X}/\!\sim \) be any open set. We must show that
\[
  \left(\overline{H}^{-1}\right)^{-1}(U') \;\subset\; h(\widetilde{\Sigma})
\]
is open. By the definition of the quotient topology and the initial topology on \( \mathcal{X} \), we have that \( q^{-1}(U') = H^{-1}(W) \) for some open set \( W \subset h(\widetilde{\Sigma}) \). This implies that
\[
  \left(\overline{H}^{-1}\right)^{-1}(U') = W,
\]
which is open in \( h(\widetilde{\Sigma}) \). Hence \( \overline{H}^{-1} \) is continuous. Because \( \overline{H} \) is a continuous bijection with a continuous inverse, it is a homeomorphism. This completes the proof.
\end{proof}
Next, we define a map $F : \mathcal{X} \longrightarrow \mathbb{H}^2$ by setting $F(y) := f(y)$ for $y$ in $ U_\alpha$ or in $C$. The following lemma shows that, since $h$ dominates $f$, the map $F$ descends to the quotient $\mathcal{X}/\sim$.
\begin{lemma}\label{quotiant_F}
The  map  $F:\mathcal{X}\to \mathbb{H}^2$ is continuous and induces a continuous map $\overline{F}:\mathcal{X}/\sim \to \mathbb{H}^2$.
\end{lemma}
\begin{proof}
Since \(h\) dominates \(f\), the inequality
\[
d_\sigma\bigl(F(x),\,F(y)\bigr)
\;\le\;
d_\sigma\bigl(H(x),\,H(y)\bigr)
\]
holds for all \(x,y\in \mathcal{X}\).  In particular, if a sequence \(x_n\to x\) in \(\mathcal{X}\), then \(H(x_n)\to H(x)\), and hence
\[
d_\sigma\bigl(F(x_n),\,F(x)\bigr)
\;\le\;
d_\sigma\bigl(H(x_n),\,H(x)\bigr)
\;\longrightarrow\; 0.
\]
This shows that \(F:\mathcal{X}\to \mathbb{H}^2\) is continuous with respect to the topology of \(\mathcal{X}\). Moreover, if \(x\sim y\), then \(H(x)=H(y)\), so
\[
d_\sigma\bigl(F(x),\,F(y)\bigr)
\;\le\;
d_\sigma\bigl(H(x),\,H(y)\bigr)
\;=\; 0,
\]
and therefore \(F(x)=F(y)\). Hence, we can define \(\overline{F}:\mathcal{X}/\sim\to \mathbb{H}^2\) by \(\overline{F}(q(x))=F(x)\).
\end{proof}
We are now in position to prove Proposition \ref{topology_of_ads_fh}.
\begin{proof}[Proof of Proposition \ref{topology_of_ads_fh}]
For each \( p \in \mathbb{H}^2 \), we denote by \( B_p \) the unique hyperbolic isometry that sends \( i \) to \( p \) and whose axis is the oriented geodesic joining \( i \) to \( p \).

To construct the homeomorphism between \( \fhads \) and \( h(\widetilde{\Sigma}) \times \mathbb{S}^1 \), we consider the map
\[
\begin{array}{ccccc}
\Phi & : & \mathcal{X} \times \mathbb{S}^1 & \longrightarrow & \mathcal{R}_{f,h} \sqcup \mathcal{S}_{f,h} \\
 & & (y, \theta) & \longmapsto & B_{H(y)} R^{\theta} B_{F(y)}^{-1}.
\end{array}
\]
Observe that \( \Phi \) is continuous--this follows from the continuity of \( H \) and \( F \). Next, note that if \( p \sim q \) in \( \mathcal{X} \), then by definition $H(p)=H(q)$ but also $F(p)=F(q)$ by the proof of Lemma \ref{quotiant_F}. This implies that \( \Phi(p, \theta) = \Phi(q, \theta) \) for any \( \theta \in \mathbb{S}^1 \). Therefore, we may define a continuous map between \( \mathcal{X}/\!\sim \times \mathbb{S}^1 \) and \( \fhads \) as follows:
\[
\begin{array}{ccccc}
\overline{\Phi} & : & \mathcal{X}/\!\sim \times \mathbb{S}^1 & \longrightarrow & \fhads \\
 & & ([y], \theta) & \longmapsto & [B_{H(y)} R^{\theta} B_{F(y)}^{-1}],
\end{array}
\]
where \( [\cdot] \) denotes the equivalence classes in both \( \mathcal{X}/\!\sim \) and \( \fhads \). We aim to show $\overline{\Phi}$ is a homeomorphism.

First, we define an inverse for $\overline{\Phi}$. To this end, we identify the circle $\mathbb{S}^1$ with the timelike geodesic $\ell_{i,i}$ via the map $e^{i\theta} \mapsto R^\theta$. Then, the inverse of $\Phi$ is given by
\[
\begin{array}{ccccc}
\Phi^{-1} & : & \mathcal{R}_{f,h} \sqcup \mathcal{S}_{f,h} & \longrightarrow & \mathcal{X} \times \mathbb{S}^1 \\
 & & A & \longmapsto & \left( \mathcal{F}(A),\; B_{H(\mathcal{F}(A))}^{-1} \, \mathcal{I}(A) \, B_{F(\mathcal{F}(A))} \right),
\end{array}
\]
where \( \mathcal{I} : \mathcal{R}_{f,h} \sqcup \mathcal{S}_{f,h} \to \ads \) and \( \mathcal{F} : \mathcal{R}_{f,h} \sqcup \mathcal{S}_{f,h} \to \widetilde{\Sigma} \) are the previously defined maps. The map \( \Phi^{-1} \) is continuous. Again, observe that if \( A \sim B \), then \( \Phi^{-1}(A) = \Phi^{-1}(B) \). This allows us to define a continuous inverse
\[
\begin{array}{ccccc}
\overline{\Phi}^{-1} & : & \fhads & \longrightarrow & \mathcal{X}/\!\sim \times \mathbb{S}^1 \\
 & & [A] & \longmapsto & \left( [\mathcal{F}(A)],\; B_{F(\mathcal{F}(A))}^{-1} \, \mathcal{I}(A) \, B_{H(\mathcal{F}(A))} \right).
\end{array}
\]
As a consequence, $\fhads$ is homeomorphic to $\mathcal{X}/\sim$, which in turn is homeomorphic to $h(\widetilde{\Sigma}) \times \mathbb{S}^1$ by Lemma \ref{lem:gluing_of_X}. It is straightforward to see that the above construction also gives rise to a homeomorphism between $\fhadss$ and $h(\widetilde{\Sigma}\setminus C)$. This completes the proof.
\end{proof}

\subsection{Anti-de Sitter structure}
We turn our attention to the regular part $\fhadss$. We will show that it admits an anti-de Sitter structure.

An anti-de Sitter structure on a three manifold $M$ is a $(\isom(\ads),\ads)$ structure. By definition, this is a maximal atlas of
coordinate charts on $M$ with values in $\ads$ such that the transition maps are given by elements of $\isom(\ads)$.

Consider the projection map
\[
\begin{array}{ccccc}
\mathrm{Pr} & : & \bigsqcup_{\alpha\in I}M_{\alpha} & \to & \fhadss=\left(\bigsqcup_{\alpha\in I}M_{\alpha}\right)\Big/\sim \\
 & & A & \mapsto & [A]. \\
\end{array}
\]
Using the topology induced on $\fhads$, we have the following.
\begin{itemize}
    \item The restriction of $\mathrm{Pr}$ to $M_{\alpha}$ is a homeomorphism onto its image.
    \item $\mathrm{Pr}(M_{\alpha})$ is an open set of $\fhads_*$.
\end{itemize}
Using this, we can show the following.

\begin{prop}
\label{prop:AdS_structure}
$\fhadss$ is an anti-de Sitter manifold with an atlas of charts
\[
  \bigl\{\,(\Pr(M_\alpha),\;\mathcal{I}\circ(\Pr|_{M_\alpha})^{-1})\bigr\}_{\alpha \in I}.
\]
Moreover, $\mathcal{F}:\fhadss\to\widetilde{\Sigma}\setminus C$ is a principal $\mathbb{S}^1$–bundle with timelike-geodesic fibers.
\end{prop}

\begin{proof}
The atlas of charts clearly defines an anti-de Sitter structure on $\fhadss$. Since 
\[
\mathcal{F} = \mathcal{F}_{\alpha} \circ \mathcal{I}\circ \left(\Pr|_{M_\alpha}\right)^{-1},
\]
and the fibers of $\mathcal{F}_{\alpha}$ are timelike geodesics, it follows that the fibers of $\mathcal{F}$ are timelike geodesics with respect to the anti-de Sitter structure induced by the atlas of charts
\[
\left\{\,\left(\Pr(M_\alpha),\; \mathcal{I}\circ \left(\Pr|_{M_\alpha}\right)^{-1}\right)\right\}_{\alpha \in I}.
\]
The above atlas of charts was introduced in~\cite[Section 3.2.3]{Agnese_thesis}. We now show that $\mathcal{F}$ is a fibration by circles, i.e., it is locally trivial. For each $\alpha \in I$, the map $\mathcal{F}_{\alpha} : M_{\alpha} \to U_{\alpha}$ is a fibration over the contractible open set $U_{\alpha}$, and thus it is trivial. Therefore, there exists a diffeomorphism
\[
\phi_{\alpha} : M_{\alpha} \to U_{\alpha} \times \mathbb{S}^1
\]
such that $\Pr_1 \circ \phi_{\alpha} = \mathcal{F}_{\alpha}$, where $\Pr_1 : U_{\alpha} \times \mathbb{S}^1 \to U_{\alpha}$ is the projection onto the first factor.
For each index $\alpha \in I$, we consider the map
\[
T_{\alpha} : \mathcal{F}^{-1}(U_{\alpha}) \to M_{\alpha}, \quad [A] \mapsto A,
\]
where $[\cdot]$ denotes the equivalence class in $\mathcal{F}^{-1}(U_{\alpha}) \subset \fhadss$. We claim that $T_{\alpha}$ is a homeomorphism.

First, for each $[A] \in \mathcal{F}^{-1}(U_{\alpha})$, we necessarily have $A \in M_{\alpha}$. Indeed, by definition, $\mathcal{F}([A]) = \mathcal{F}_{\alpha}(A) \in U_{\alpha}$, hence $A(f(p)) = h(p)$, that is, $A \in M_{\alpha}$. Second, the map is well-defined. Indeed, if $[A] = [B]$ with $[A], [B] \in \mathcal{F}^{-1}(U_{\alpha})$, then we have seen that $A, B \in M_{\alpha}$, and because $A$ and $B$ are equivalent, we have $\mathcal{I}(A) = \mathcal{I}(B)$. Since, $\mathcal{I}|_{M_{\alpha}}$ is the inclusion map, we deduce $A = B$.

Next, the inverse of $T_{\alpha}$ is clearly given by
\[
T_{\alpha}^{-1} : M_{\alpha} \to \mathcal{F}^{-1}(U_{\alpha}), \quad A \mapsto [A].
\]
The continuity of both $T_{\alpha}$ and $T_{\alpha}^{-1}$ follows from the definition of the topology on $\fhadss$.
Now, since $T_{\alpha} \circ \phi_{\alpha} : \mathcal{F}^{-1}(U_{\alpha}) \to U_{\alpha} \times \mathbb{S}^1$ is a local trivialization of $\mathcal{F}$, this completes the proof.
\end{proof}
The proof of Theorem \ref{ads_manifold} follows by combining Propositions \ref{topology_of_ads_fh} and \ref{prop:AdS_structure}.\\

We end this section by studying the case where the maps $f,h : \widetilde{\Sigma} \to \mathbb{H}^2$ are equivariant with respect to the representations $\rho$ and $j : \Gamma \to \psl$, respectively. In what follows, we denote
\[
\Gamma_{j,\rho} = \{(j(\gamma),\rho(\gamma)) : \gamma \in \Gamma\}.
\]
We will show that $\Gamma_{j,\rho}$ acts properly discontinuously on both $\fhads$ and $\fhadss$. In particular, it will follow from Proposition~\ref{prop:AdS_structure} that the quotient of $\fhadss$ is an anti–de Sitter manifold.

First, we describe this action. For each $\gamma \in \Gamma$, let $H(\gamma) = (j(\gamma),\rho(\gamma))$. Then $H(\gamma)$ induces a map
\begin{equation} \label{eq_action_gluing-manifold}
\overline{H}(\gamma) : \fhads \to \fhads
\end{equation}
defined by
\begin{equation}\label{eq: H action}
    \overline{H}(\gamma)\cdot [A] := [\, j(\gamma)\,\mathcal{I}(A)\,\rho(\gamma)^{-1} \,].
\end{equation}
It is straightforward to verify that the map $\mathcal{F} : \fhads \to \widetilde{\Sigma}$ satisfies the following equivariance property: for all $\gamma \in \Gamma$ and $[A] \in \fhads$,
$$
\mathcal{F}\!\left(\overline{H}(\gamma) \cdot [A]\right) = \gamma \cdot \mathcal{F}([A]).
$$
Since $\mathcal{F}$ is continuous and equivariant, and $\Gamma$ acts properly discontinuously on its universal cover $\widetilde{\Sigma}$, it follows that the action of $\Gamma_{j,\rho}$ on $\fhads$ is properly discontinuous.

On the other hand, $\fhads$ is a Hausdorff topological space by Proposition~\ref{topology_of_ads_fh}, and thus the quotient $\fhads / \Gamma_{j,\rho}$ is a well–defined Hausdorff topological space. Moreover, the quotient $\fhadss / \Gamma_{j,\rho}$ inherits an anti–de Sitter structure. We summarize this discussion in the following lemma.

\begin{lemma}\label{Mfh}
Let $f, h : \widetilde{\Sigma} \to \mathbb{H}^2$ be maps equivariant with respect to the representations $j, \rho : \Gamma \to \mathrm{PSL}(2,\mathbb{R})$, respectively, and satisfying conditions \ref{cond:h_singular}--\ref{cond:h_dominates}. Then the map $\mathcal{F} : \fhads \to \widetilde{\Sigma}$ is equivariant with respect to the action of $\Gamma_{j,\rho}$ on $\fhads$ via $\overline{H}$ (see \eqref{eq_action_gluing-manifold}) and the action of $\Gamma$ on $\widetilde{\Sigma}$. In particular, $\Gamma_{j,\rho}$ acts properly discontinuously on $\fhads$, and the quotient
\[
\mathcal{M}^{f,h} := \fhads \big/ \Gamma_{j,\rho}
\]
is a well-defined Hausdorff topological space. Moreover, the map $\mathcal{F}$ induces a principal $\mathbb{S}^1$-bundle over the anti-de Sitter manifold
\[
\mathcal{M}_*^{f,h} := \fhadss \big/ \Gamma_{j,\rho},
\]
with base $\Sigma_*$, where $\Sigma_*$ is the quotient of $\widetilde{\Sigma}\setminus C$ by the action of $\Gamma$. We continue to denote this map by 
\[
\mathcal{F} : \mathcal{M}_*^{f,h} \to \Sigma_*.
\]
\end{lemma}

\section{Branched AdS manifolds from branched immersions}\label{sec6}
In this section, we study the anti-de Sitter manifold obtained from a pair $f,h$ as above and under the additional assumption that $h$ is a branched immersion. Recall that a map $h$ between surfaces has a \emph{branch point} at $p$ if there exist local complex coordinates $z$ at $p$ and $w$ at $h(p)$ such that $h$ takes the form $z \mapsto w = z^n$. We say that $h$ is a \emph{branched immersion} if it is an immersion outside a discrete set of branch points. The main result of this section is the proof of Theorem \ref{principal_ads}, which is completed at the end.

Our approach follows, in spirit, the arguments developed by Janigro in her thesis \cite{Agnese_thesis}, within her framework for analyzing local singularities of anti–de Sitter $3$-manifolds. To extend her framework to our more general setting (outlined in Section \ref{sec5}), we introduce several technical refinements and additional results, which are established in the course of the proof of Proposition \ref{fonda_ex}.

\subsection{The fundamental example}\label{sec_fonda_ex}
Before proving Theorem \ref{principal_ads}, we first treat the special case where $\Sigma = \mathbb{H}^2$ and the map $h : \widetilde{\Sigma} \to \mathbb{H}^2$ has a single branch point. Understanding this example is an essential step toward the proof of the general theorem. We consider smooth maps $f,h : \mathbb{H}^2 \to \mathbb{H}^2$ such that $h$ dominates $f$ (i.e.\ $f^*\sigma < h^*\sigma$). Without loss of generality, by composing with isometries, we may assume that $h(i) = f(i) = i$.
We have seen in Section \ref{sec5} how to construct an anti-de Sitter manifold from the maps $f$ and $h$. Recall that

$$\fhads := \left( \bigsqcup_{\alpha \in I} M_{\alpha} \sqcup \ell_{i,i} \right) \Big/ \sim ,$$ where $M_{\alpha} := \left\{ A \in \psl \;\middle|\; \exists! \, x \in U_{\alpha} \text{ such that } A(f(x)) = h(x) \right\},$ and $\{U_{\alpha}\}_{\alpha \in I}$ is an open cover of $\mathbb{H}^2_* := \mathbb{H}^2 \setminus \{i\}$ on which $h$ is a local diffeomorphism. Note that in this case, $M_{\alpha}$ does not contain the timelike geodesic $\ell_{i,i}$, because $f$ and $h$ take values in $\mathbb{H}^2_*$. From Proposition \ref{prop:AdS_structure}, we can define the local isometry
\begin{equation}\label{D_cal}
\begin{array}{rcl}
\mathcal{D} & : & \fhadss \to \ads \\
             &   & [A] \mapsto \mathcal{I}(A).
\end{array}
\end{equation}
Note that $\mathcal{D}$ takes values in $\ads_* := \ads \setminus \ell_{i,i}$. Moreover, $\mathcal{D}$ can also be defined on all of $\fhads$ by setting $\mathcal{D}(A) = \mathcal{I}(A)$ for any $[A] \in \ell_{i,i}$. We record here some notations and remarks that will be used later.
\begin{itemize}
    \item We denote by \(\widetilde{\mathrm{Dev}}: \widetilde{\fhadss} \to \widetilde{\ads_*}\) the developing map of \(\fhadss\). This can be thought of as the lift to the universal cover of the local isometry \(\mathcal{D} : \fhads \to \ads_*\).

    \item We denote by \(\mathrm{Dev} : \widetilde{\fhadss} \to \ads_*\) the developing map onto \(\ads_*\). This is given by
    \begin{equation}\label{eq_dev_on_ads}
        \mathrm{Dev} = T \circ \widetilde{\mathrm{Dev}},
    \end{equation}
    where \(T : \widetilde{\ads_*} \to \ads_*\) is the universal covering map defined in \eqref{agnesmap}. 

    \item Finally, we denote by \(\widetilde{\mathrm{Hol}} : \pi_1(\fhadss) \to \isom(\widetilde{\ads_*})\) the holonomy representation of \(\fhadss\).
\end{itemize}
According to Proposition~\ref{prop:AdS_structure}, there exists an $\mathbb{S}^1$–principal bundle $\mathcal{F} : \fhadss \longrightarrow \mathbb{H}^2_*.$ This bundle is trivial. Indeed, one can obtain a global trivialization directly from the proof of Proposition~\ref{topology_of_ads_fh}. Alternatively, since any map from a punctured disc to the classifying space $\mathbb{CP}^\infty$ is nullhomotopic, any $\mathbb{S}^1$–bundle over a punctured disc must be trivial.

In what follows, we denote by $\alpha$ a nontrivial loop in $\pi_1(\fhadss)$ that is not homotopic to the fiber of $\mathcal{F}$ and whose projection under $\mathcal{F}$ is a generator $\gamma$ of $\pi_1(\hh)\cong \mathbb{Z}$. (Here, $\gamma$ may be viewed as a loop around the branched point of $\hh$.) We also denote by $\beta$ the generator of $\pi_1(\fhadss)$ that is homotopic to the future-directed timelike geodesic fiber of the fibration $\mathcal{F} : \fhadss \to \mathbb{H}^2_*$. Since the bundle $\mathcal{F}$ is trivial, the fundamental group of $\fhadss$ decomposes as
\begin{equation}\label{group_fonda_M}
\pi_1(\fhadss) \;=\; \mathbb{Z}\alpha \;\oplus\; \mathbb{Z}\beta.
\end{equation}
The spacetime $\fhadss$ is time-oriented, so that $1 \in \mathbb{Z}$ corresponds to the future-directed fiber, while $-1$ corresponds to the same fiber with the opposite orientation.

We can now state the main result of this section, which is inspired by \cite[pp.~61–69]{Agnese_thesis}.

\begin{prop}\label{fonda_ex}
We consider two smooth maps $f,h : \mathbb{H}^2 \to \mathbb{H}^2$ such that $h$ dominates $f$. Assume that $h(i)=f(i)=i$ and that $h:\mathbb{H}^2_*\to \mathbb{H}^2_*$ is a degree-$n$ covering. Then $\fhadss$ is a branched anti-de Sitter manifold with singular locus $[\ell_{i,i}]$. Moreover,
$$ \widetilde{\mathrm{Hol}}(\alpha) = \varphi_{(2n\pi,  2k\pi)} \ \text{and} \ \ \ 
   \widetilde{\mathrm{Hol}}(\beta) = \varphi_{(0, 2\pi)},$$ for some $k\in \mathbb{Z}.$ Therefore, the developing map \( \widetilde{\mathrm{Dev}}: \widetilde{\fhadss} \to \widetilde{\ads_*} \) induces a local isometry between \( \fhadss \) and \( \ads_{(2n\pi, 0)} \).
\end{prop}
\begin{remark}
  It is worth noting that if we change the meridian $\alpha$ by $\alpha'=\alpha-k\beta$, then the holonomy of $\alpha'$ is $\varphi_{(2n\pi,0)}$. Thus, what really matters in the holonomy of the meridian is the degree of $h$, and this explains why we do not require any further assumption on the map $f$.
\end{remark}

\begin{remark}\label{local_remark}
The statement of Proposition \ref{fonda_ex} remains valid if the maps \( f \) and \( h \) are defined on a small neighborhood of \( i \) in \( \mathbb{H}^2 \) instead of the whole plane. The proof carries over verbatim in this local setting. 
\end{remark}

We record the following lemma for later use.
\begin{lemma}\label{week_contraction_around_singular_point}
Let $f, h : \mathbb{H}^2 \to \mathbb{H}^2$ be as in Proposition \ref{fonda_ex}.  
Then, for all $x \in \mathbb{H}^2$, we have the strict inequality
\[
d_{\sigma}(i, f(x)) < d_{\sigma}(i, h(x)).
\]  
In particular, the timelike geodesics $\ell_{i,i}$ and $\ell_{h(x), f(x)}$ do not intersect. 
\end{lemma}
\begin{proof}
The domination condition $f^*\sigma \leq h^*\sigma$ implies that
\[
\ell(f \circ \gamma) \leq \ell(h \circ \gamma)
\]
for every curve $\gamma$ in $\mathbb{H}^2$, where $\ell(\cdot)$ denotes the length of the curve with respect to the hyperbolic metric $\sigma$. We claim that $d_{\sigma}(i,f(x))\leq d_{\sigma}(i,h(x))$ for all $x \in \mathbb{D}^2_*$. For simplicity, we identify $\mathbb{H}^2$ with the Poincaré disc $\mathbb{D}^2$ via an identification sending the point $i$ to the origin $0$ (see~\eqref{identification}). Assume that $h(x) = r_0 e^{i\theta_0}$ for some $r_0, \theta_0 \in \mathbb{R}$. Consider the geodesic ray $\eta : (0,1) \to \mathbb{D}^2_*$ defined by $\eta(r) = r e^{i\theta_0}$. By the path lifting property, there exists a curve $\gamma : (0,1) \to \mathbb{D}^2_*$ such that $h \circ \gamma = \eta$ and $\gamma(r_0) = x$. This implies that for each $0 < t < r_0$, we have
\begin{align*}
d_{\sigma}(f(\gamma(t)), f(x)) 
&\leq \ell(f \circ \gamma|_{[t, r_0]}) \\
& \leq \ell(h \circ \gamma|_{[t, r_0]}) \\
&= \ell(\eta|_{[t, r_0]}) = d_{\sigma}(h(\gamma(t)), h(x)).
\end{align*}
To establish the claim, we need to show that $\gamma(t) \to 0$ as $t \to 0$. Indeed, for $t < r_0$, we have $\gamma(t) \in h^{-1}(\{z \in \mathbb{D}^2 \mid |z| \leq r_0\})$, which is a compact subset of $\mathbb{D}^2$ by the properness of $h$. To see properness, recall that $h$ is equal to $z^n$ in an appropriate choice of local coordinates around $0$. Therefore, $\gamma(t)$ must converge to some point $x \in \mathbb{D}^2$. But since $h(\gamma(t)) = \eta(t)$ and $\eta(t) \to 0$ as $t \to 0$, it follows that $\gamma(t) \to 0$. This completes the proof of the claim. To prove the strict inequality, consider \( y \) a point on the geodesic segment between \( i \) and \( x \). Recall that domination is strict outside the singularities. Hence, for $y$ sufficiently close to $x$, we have $d_{\sigma}\bigl(f(x),f(y)\bigr) < d_{\sigma}\bigl(h(x),h(y)\bigr)$.
It follows that
\[
d_{\sigma}\bigl(f(x), i\bigr) 
\leq d_{\sigma}\bigl(f(x), f(y)\bigr) + d_{\sigma}\bigl(f(y), i\bigr) 
< d_{\sigma}\bigl(h(x), h(y)\bigr) + d_{\sigma}\bigl(h(y), i\bigr) 
= d_{\sigma}\bigl(h(x), i\bigr).
\]
This concludes the proof.\end{proof}

\begin{remark}
We denote by $[\ell_{i,i}]$ the set of classes in \( \fhads \) that are equivalent to some element lying on \( \ell_{i,i} \). That is,
$$
[\ell_{i,i}] = \left\{ [A] \in \fhads \;\middle|\; 
\mathcal{F}(A) = i \ \text{and} \ \mathcal{I}(A) \in \ell_{i,i} \right\}.
$$
Observe that if \( [A] \in [\ell_{i,i}] \), then \( A \) cannot belong to any \( M_{\alpha} \), since this would contradict Lemma \ref{week_contraction_around_singular_point}.
\end{remark}

The next lemma shows that the developing map of \( \fhadss \) satisfies Definition~\ref{defi_spin_cone}.

\begin{lemma}\label{extension_of_dev}
The developing map $\mathrm{Dev} : \widetilde{\fhadss} \to \ads_*$ extends continuously to the universal branched cover $\widetilde{\fhadss} \sqcup \widetilde{[\ell_{i,i}]}$. Moreover, the restriction 
\[
\mathrm{Dev} \big|_{\widetilde{[\ell_{i,i}]}} : \widetilde{[\ell_{i,i}]} \to \ell_{i,i}
\]
is a universal covering map of the timelike geodesic $\ell_{i,i}\cong\mathbb{S}^1$.
\end{lemma}

\begin{proof}
Let $\Pi : \widetilde{\fhadss} \to \fhadss$ be the universal covering map. By definition, the developing map is given by $\mathrm{Dev} = \mathcal{D} \circ \Pi$, where \( \mathcal{D} : \fhads \to \ads \) is defined by \( \mathcal{D}([A]) = \mathcal{I}(A) \), as in equation \eqref{D_cal}.
By Proposition \ref{topology_of_ads_fh}, \( \fhads \) is a solid torus. Therefore, the universal covering map $\Pi$ can be described using the cylindrical coordinates \eqref{cylindrical}, and hence \( \Pi \) extends continuously to \( \fhadss \sqcup \widetilde{[\ell_{i,i}]} \). Since \( \mathcal{D} \) is already defined on all of \( \fhads \) (and not only on \( \fhadss \)), it follows that \( \mathrm{Dev} \) also extends to \( \widetilde{\fhadss} \sqcup \widetilde{[\ell_{i,i}]} \).

For the second claim, note that \( \mathcal{D} \big|_{[\ell_{i,i}]} : [\ell_{i,i}] \to \ell_{i,i} \) is a homeomorphism, and \( \mathrm{Dev} = \mathcal{D} \circ \Pi \). Since \( \Pi \) restricts to the universal cover of \( [\ell_{i,i}] \), the composition \( \mathrm{Dev} \big|_{\widetilde{[\ell_{i,i}]}} \) is the universal covering map of \( \ell_{i,i} \).
\end{proof}

\subsection{Holonomy around the puncture in the fundamental example}
The aim of this section is to compute the holonomy of a nontrivial loop in $\pi_1(\fhadss)$ that is not homotopic to a fiber of the fibration $\mathcal{F} : \fhadss \to \mathbb{H}^2_*$ and that generates the $\pi_1$ of the base. As in Section \ref{sec_fonda_ex}, we denote such a loop by $\alpha$, and we denote by $\gamma$ its projection to $\mathbb{H}^2_*$, which generates $\pi_1(\mathbb{H}^2_*)$. We now prove the following.
\begin{prop}\label{holonomy_around_puncture}
Let $f$ and $h$ be as in Proposition \ref{fonda_ex}. Then we have $$\widetilde{\mathrm{Hol}}(\alpha) = \varphi_{(2n\pi, 2k\pi)},$$ for some $k\in \mathbb{Z}$.
\end{prop}
The proof of Proposition \ref{holonomy_around_puncture} needs some preparation.
\begin{defi}\label{angular_form}
Let \( \pi: (0, +\infty) \times \mathbb{R} \to \mathbb{H}^2_* \) be the universal covering map of \( \mathbb{H}^2_* \) defined in \eqref{eq_pi_map}. Then we define the \textit{angular form} \( \omega \) as the $1$-form on \( \mathbb{H}^2_* \)
\[
\pi^*\omega = \mathrm{d}\theta.
\]
\end{defi}

\begin{remark}
We record the following observations.
\begin{enumerate}
    \item If we consider the biholomorphism \( \mathcal{C}: \mathbb{D}^2 \to \mathbb{H}^2 \) defined by
  \begin{equation}\label{identification}
    \mathcal{C}: z \mapsto i \frac{z - i}{z + i}, \quad \text{which maps} \quad i \quad \text{to} \quad 0,
    \end{equation}
    then the angular $1$-form \( \omega \) has the following expression in \( \mathbb{D}^2 \setminus \{0\} \)
    \[
    \mathcal{C}^*\omega = \frac{-y}{x^2 + y^2} \, dx + \frac{x}{x^2 + y^2} \, dy.
    \]
    \item The angular form \( \omega \) is preserved by any rotation \( R^{\theta} \). That is,
    \[
    (R^{\theta})^*\omega = \omega.
    \]
\end{enumerate}
\end{remark}
The following elementary and well-known lemma expresses the degree of a branched covering in terms of the angular form $\omega$.
\begin{lemma}\label{theta_as_integral}
Let \( \gamma: [0, 1] \to \hh \) be a curve around the puncture generating \( \pi_1(\hh) \), and let \( h: \hh \to \mathbb{H}^2_* \) be a degree-$n$ cover branched on $i$. Then,
\[
\int_{h\circ\gamma} \omega = 2n\pi.
\]
\end{lemma}

\begin{proof}
Since $h$ is a degree-$n$ cover of $\hh$, $\int_\gamma h^*\omega=n\int_\gamma \omega$. Hence,
\[
\int_{h \circ \gamma} \omega = \int_{ \gamma} h^*\omega = n\int_{  \gamma}\omega.
\]
Now it is enough to take a parametrization of a representative of a loop around the puncture and compute $\int_\gamma \omega$. This is possible as the integral depends only on the homotopy class of \( \gamma \). So we take \( \gamma(t) = R^{2\pi t}c(r_0) \in \mathbb{H}^2_* \) (see \eqref{eq_pi_map}). We then obtain \( \int_{\gamma} \omega = 2\pi \), which yields the desired conclusion.
\end{proof}
The next lemma gives an expression for a lift of a curve in $\ads_*$ to $\widetilde{\ads_*}$
\begin{lemma}\label{Curve_lift_in_ads}
    Let $T:(0, +\infty) \times \mathbb{R} \times \mathbb{R}  \to \ads_*$ be the universal covering map of $\ads_*$ defined in \eqref{agnesmap}.
    Let $c: \mathbb{R} \to \ads_*$ be a curve in $\ads_*$ and let $\widetilde{c}: \mathbb{R} \to \widetilde{\ads_*}$ be a lift of the curve $c$ such that $\widetilde{c}(0) = (r_0, x_0, y_0)$. Then, 
    \[
    \widetilde{c}(t) = \left(d_{\sigma}(i, c(t) \cdot i), x_0 + \int_{c([0,t])}\mathrm{Val}^*(\omega), y_0 + \int_{c([0,t])} \mathrm{\overline{Val}}^*(\omega)\right),
    \]
    where $\mathrm{Val},\overline{\mathrm{Val}}:\ads_*\to \mathbb{H}^2_*$ are the valuations map defined by  $\mathrm{Val}(A) = A(i)$ and $\overline{\mathrm{Val}}(A) = A^{-1}(i)$.
\end{lemma}

\begin{proof}
    Assume that $\widetilde{c}(t) = (r(t), \theta(t), \eta(t))$, so that 
    \[
    c(t) = R^{\theta(t)} A(r(t)) R^{-\eta(t)}.
    \]
    The equality $r(t) = d_{\sigma}(i, c(t) \cdot i)$ follows easily. We now focus on deriving the formula for $\theta(t)$ using the angular form $\omega$. Observe that
    \begin{equation}\label{eq_val}
        \mathrm{Val}(c(t)) = \pi(r(t), \theta(t)),
    \end{equation}
    where $\pi:(0,+\infty)\times\mathbb{R}\to \mathbb{H}^2_*$ is the map from \eqref{eq_pi_map}. We compute,
    
    \begin{align*}
        \int_{c([0,t])} \mathrm{Val}^*(\omega) 
        &= \int_0^t \omega_{\mathrm{Val}(c(s))}\left(\mathrm{d}_{c(s)} \mathrm{Val}(c'(s))\right) \mathrm{d}s \\
        &= \int_0^t \omega_{\pi(r(s), \theta(s))}\left(\mathrm{d}_{(r(s), \theta(s))} \pi(r'(s), \theta'(s))\right) \mathrm{d}s \\
        &= \int_0^t \mathrm{d}\theta(r'(s), \theta'(s)) \mathrm{d}s \\
        &= \int_0^t \theta'(s) \mathrm{d}s \\
        &= \theta(t) - x_0,
    \end{align*}
    where in the second equality, we used Equation \eqref{eq_val}, and in the third equality, we used Lemma \ref{angular_form}. The proof for the third coordinate of $\widetilde{c}$ is essentially the same. This completes the proof.
\end{proof}

The final lemma toward Proposition \ref{holonomy_around_puncture} is below. The proof follows line by line the argument of \cite[Lemma~3.5.3]{Agnese_thesis}.

\begin{lemma}\cite[Lemma~3.5.3]{Agnese_thesis}\label{Near_curve}
 Let \( \gamma_1 : \mathbb{R} \to \mathbb{H}^2_* \) and \( \gamma_2 : \mathbb{R} \to \mathbb{H}^2_* \) be two paths in \( \mathbb{H}^2_* \) satisfying \( \gamma_j(t + t_0) = R^{\theta_0} \gamma_j(t) \), for some \( t_0, \theta_0 \in \mathbb{R} \). Assume that for all \( t \in \mathbb{R} \), the geodesic segment $[\gamma_1(t),\gamma_2(t)]$ does not contain $i$. Then,
    \[
    \int_{\gamma_1([0, t_0])} \omega = \int_{\gamma_2([0, t_0])} \omega.
    \]
\end{lemma}

We now turn to the proof of Proposition \ref{holonomy_around_puncture}. Recall that $\alpha$ is a nontrivial loop in $\pi_1(\fhads_*)$ that is not homotopic to the fiber of $\mathcal{F}$, and that $\alpha$ projects to a loop $\gamma$ around the puncture of $\mathbb{H}^2_{*}$ that generates the $\pi_1$. Let \( \widetilde{\alpha} \) be a lift of $\alpha$ in \( \widetilde{\fhads_*} \). For every \( t \in \mathbb{R} \), we denote
$$ \widetilde{A}_t := \widetilde{\text{Dev}}(\widetilde{\alpha}(t)) $$ 
and 
$$ A_t := \mathrm{Dev}(\widetilde{\alpha}(t)) = T(\widetilde{\mathrm{Dev}}(\widetilde{\alpha}(t)))$$ (see \eqref{eq_dev_on_ads}). Consequently, we have the following.
\begin{itemize}
\item By the construction of the manifold \( \fhadss \), the path \( A_t \) in \( \ads_* \) has the property
    \[
    A_t(f(\gamma(t))) = h(\gamma(t))
    \]
    for every \( t \). 
    \item By definition of the action of the fundamental group, we have
    \begin{equation}\label{eq_deck_transformation}
    \widetilde{\text{Dev}}(\widetilde{\alpha}(1)) = \widetilde{\text{Dev}}(\alpha \cdot \widetilde{\alpha}(0)) = \widetilde{H}(\alpha) \widetilde{\text{Dev}}(\widetilde{\alpha}(0)),
    \end{equation}
    where $\widetilde{H}(\alpha)=\varphi_{(\theta_0,\eta_0)}$ is the holonomy of $\alpha$, see \eqref{isometry_ads}.

\end{itemize}
Before proving Proposition \ref{holonomy_around_puncture}, we need a basic lemma on hyperbolic geometry. For each $p\in \mathbb{H}^2_*$, we denote by \( B_p \) the unique hyperbolic isometry that sends \( i \) to \( p \) and whose axis is the oriented geodesic joining \( i \) to \( p \).
\begin{lemma}\label{Up_Lemma}
    Let $p,q\in \mathbb{H}^2_*$ be such that $d_{\sigma}(i,p)\neq d_{\sigma}(i,q)$. Then, $B_pB_q^{-1}$ does not fix $i$.
\end{lemma}

\begin{proof}
    By contradiction, assume that $R:=B_pB_q^{-1}$ fixes $i$. By definition, $Rq=p$, and hence $d_{\sigma}(i,q)=d_{\sigma}(Ri,Rp)=d_{\sigma}(i,p)$, which is a contradiction.
\end{proof}

\begin{proof}[Proof of Proposition \ref{holonomy_around_puncture}]
In this proof, we take explicit $\gamma$ and $\alpha$ in order to make computations.  
We set $\gamma(t) = re^{2i\pi t}$ and 
\[
A_t = B_{h(\gamma(t))} B^{-1}_{f(\gamma(t))},
\]
where \( B_p \) the unique hyperbolic isometry that sends \( i \) to \( p \) and whose axis is the oriented geodesic joining \( i \) to \( p \). Since $\gamma(t+1) = \gamma(t)$, it follows that $A_{t+1} = A_t$, and therefore $ \alpha(t) = [A_t] $ is a loop in $\fhadss$ that is not homotopic to a fiber of the fibration $\mathcal{F} : \fhadss \to \mathbb{H}^2_*$.

The goal is now to compute the holonomy of the loop $\alpha$. Let $\widetilde{\alpha}$ be the lift of $\alpha$ to $\widetilde{\fhadss}$.  
By definition, we have 
\[
\mathrm{Dev}(\widetilde{\alpha}(t)) = \mathcal{D}(\alpha(t)) = A_t,
\]
and so $\widetilde{\mathrm{Dev}}(\widetilde{\alpha}(t))$ is the lift to $\widetilde{\ads_*}$ of $\mathrm{Dev}(\widetilde{\alpha}(t))$.  
Consider $\theta_0,\eta_0\in \mathbb{R}$ such that $\widetilde{\mathrm{Hol}}(\alpha) = \varphi_{(\theta_0,\eta_0)}$ (see notation in \eqref{isometry_ads}).  
Thus, by Proposition \ref{Curve_lift_in_ads}, we have
\begin{equation}\label{lift_curve_dev}
\widetilde{\mathrm{Dev}}(\widetilde{\alpha}(t)) 
= \left( \mathrm{d}_{\mathbb{H}^2}\!\left(i, \mathrm{Dev}(\widetilde{\alpha}(t)) \cdot i\right),\ 
x_0 + \int_{\mathrm{Dev \circ \widetilde{\alpha}}([0,t])} \mathrm{Val}^*(\omega),\ 
y_0 + \int_{\mathrm{Dev} \circ \widetilde{\alpha}([0,t])} \overline{\mathrm{Val}}^*(\omega) \right),
\end{equation}
for some $x_0,y_0 \in \mathbb{R}$. This leads to
\begin{equation}\label{eq_proof}
\begin{aligned}
\widetilde{\mathrm{Dev}}(\widetilde{\alpha}(1)) 
&= \left( \mathrm{d}_{\sigma}\!\left(i, \mathrm{Dev}(\widetilde{\alpha}(1)) \cdot i\right),\ 
x_0 + \int_{\mathrm{Dev \circ \widetilde{\alpha}}([0,1])} \mathrm{Val}^*(\omega),\ 
y_0 + \int_{\mathrm{Dev} \circ \widetilde{\alpha}([0,1])} \overline{\mathrm{Val}}^*(\omega) \right).
\end{aligned}
\end{equation}

We claim that
\begin{equation}\label{eq_claim}
    \int_{\mathrm{Dev} \circ \widetilde{\alpha}([0,1])} \mathrm{Val}^*(\omega) = 2n\pi.
\end{equation}

To prove the claim, let us consider the paths 
\[
\gamma_1(t) := A_t(i), \qquad 
\gamma_2(t) := h(\gamma(t)) = A_t(f(\gamma(t))).
\]
Clearly $\gamma_2$ is a loop in $\mathbb{H}^2_*$, and the same holds for $\gamma_1$ by Lemma \ref{Up_Lemma}, which ensures that $\gamma_1(t) \in \mathbb{H}^2_*$ for all $t \in \mathbb{R}$. Next, we want to apply Lemma \ref{Near_curve}. The curves $\gamma_j$ satisfy $\gamma_j(t+1) = \gamma_j(t)$ for $j=1,2$ (because $\gamma(t+1)=\gamma(t)$ and $A_{t+1}=A_t$). By Lemma \ref{week_contraction_around_singular_point}, we have
\[
d_{\sigma}(\gamma_1(t), \gamma_2(t)) = d_{\sigma}(i, f(\gamma(t))) < d_{\sigma}(i, h(\gamma(t))) = d_{\sigma}(i, \gamma_2(t)).
\] This implies in particular that $i$ does not belong to the geodesic segment $[\gamma_1(t),\gamma_2(t)]$. Therefore,
\begin{equation}\label{computation}
\int_{\mathrm{Dev}\circ\widetilde{\alpha}([0,1])}\mathrm{Val}^*(\omega)
=\int_{\gamma_2}\omega
=\int_{h\circ\gamma}\omega
=2n\pi,
\end{equation}
where the last equality follows from the fact that $h$ is a branched covering of degree $n$ (see Lemma \ref{theta_as_integral}). Next, since
\[
\widetilde{\mathrm{Dev}}(\widetilde{\alpha}(1))
= \varphi_{(\theta_0,\eta_0)}\left(\widetilde{\mathrm{Dev}}(\widetilde{\alpha}(0))\right),
\]
it follows from \eqref{lift_curve_dev} that
\[
x_0 + \int_{\mathrm{Dev}\circ\widetilde{\alpha}([0,1])}\mathrm{Val}^*(\omega)
= x_0 + \theta_0.
\]
Combined with \eqref{computation}, we obtain $\theta_0 = 2n\pi,$ which finishes the proof of the claim.  
  
The remaining part is to compute $$\int_{\mathrm{Dev} \circ \widetilde{\alpha}([0,1])} \overline{\mathrm{Val}}^*(\omega),$$ which is equal to $\int_{\gamma_3}\omega$, where $\gamma_3(t) := A_t^{-1}\cdot i$. Since $\gamma_3$ is a closed loop, it follows from classical degree theory and the calculation at the end of Lemma \ref{theta_as_integral} that this last integral equals $2k\pi$ for some $k \in \mathbb{Z}$. This concludes the proof, as
\[
\widetilde{\mathrm{Dev}}(\widetilde{\alpha}(1)) = \widetilde{H}(\alpha)\, \widetilde{\mathrm{Dev}}(\widetilde{\alpha}(0))
\]
and
\begin{equation}
\begin{aligned}
\widetilde{\mathrm{Dev}}(\widetilde{\alpha}(1)) 
&= \left( \mathrm{d}_{\sigma}\!\left(i, \mathrm{Dev}(\widetilde{\alpha}(0)) \cdot i\right),\ 
x_0 + 2n\pi,\ 
y_0 + 2k\pi \right).
\end{aligned}
\end{equation}
\end{proof}
\subsection{Holonomy of the fiber in the fundamental example}
The main goal of this subsection is to prove the following Proposition.

\begin{prop}\label{holonomy_of_the_fiber}
    Let \(\beta\) be a generator of \(\pi_1(\fhadss)\) that is homotopic to a future-directed timelike geodesic fiber of \(\mathcal{F} : \fhadss \to \hh\) (see \eqref{group_fonda_M}). Then,
    \[
    \widetilde{\mathrm{Hol}}(\beta) = \varphi_{(0, 2\pi)}.
    \]
\end{prop}
The proof of this proposition proceeds through the study of the stabilizers of timelike geodesics in $\widetilde{\ads_*}$.  
For $x, y \in \mathbb{H}^2$, let $\ell_{x,y}$ be the timelike geodesic in $\ads$ defined in \eqref{eq_timelike}.  
If $\ell_{x,y}$ belongs to $\ads_{*}$, we denote by $\widetilde{\ell}_{x,y}$ its lift to $\widetilde{\ads_{*}}$. We denote by $\mathrm{Stab}(\ell_{x,y})$ the stabilizer of $\ell_{x,y}$ in $\mathrm{Isom}(\ads)$, and by $\mathrm{Stab}(\widetilde{\ell}_{x,y})$ the stabilizer of $\widetilde{\ell}_{x,y}$ in $\mathrm{Isom}(\widetilde{\ads_*})$. That is,
\[
\mathrm{Stab}(\ell_{x,y}) :=  \left\{ \varphi \in \mathrm{Isom}(\ads)\ \mid\ \varphi(\ell_{x,y}) = \ell_{x,y} \right\}
\]
and
\[
\mathrm{Stab}(\widetilde{\ell}_{x,y}) :=  \left\{ \varphi \in \mathrm{Isom}(\widetilde{\ads_*})\ \mid\ \varphi(\widetilde{\ell}_{x,y}) = \widetilde{\ell}_{x,y} \right\}.
\]
We have the following lemma.

\begin{lemma}\cite[Propositions~2.2.5 and~2.2.6]{Agnese_thesis}\label{stab_timelike}
    Let \(x, y \in \mathbb{H}^2\) be such that \(d_{\sigma}(i, x) < d_{\sigma}(i, y)\). Consider the timelike geodesic \(\ell_{x,y}\), which is disjoint from \(\ell_{i,i}\), and let \(\widetilde{\ell_{x,y}}\) be its lift in \(\widetilde{\ads_*}\). Then,
    \[
    \mathrm{Stab}(\widetilde{\ell_{x,y}}) = \{0\} \times 2\pi\mathbb{Z}.
    \]
\end{lemma}

\begin{proof}[Proof of Proposition \ref{holonomy_of_the_fiber}]
    Let \(x \in \mathbb{H}^2_*\). By Lemma \ref{week_contraction_around_singular_point},
    \[
    d_{\sigma}(i, f(x)) < d_{\sigma}(i, h(x))
    \]
and the timelike geodesics \(\ell_{h(x),f(x)}\) and \(\ell_{i,i}\) are distinct. In particular, \(\ell_{h(x),f(x)} \subset \ads_*\).  
Let \(\beta\) be a generator of an oriented fiber of the fibration \(\mathcal{F} : \fhadss \to \mathbb{H}^2_*\). We may choose a representative of \(\beta\) so that \(\mathcal{D}(\beta)\) is the oriented timelike geodesic \(\ell_{h(x),f(x)}\) (see \eqref{D_cal}). Let $\widetilde{\beta}$ be the lift of $\beta$ to $\widetilde{\fhadss}$ that is preserved by the deck transformation induced by \(\beta\).  
By the equivariance of the developing map \(\widetilde{\mathrm{Dev}} : \widetilde{\fhadss} \to \widetilde{\ads_*}\), the holonomy \(\widetilde{\mathrm{Hol}}(\beta)\) preserves the timelike geodesic $\widetilde{\mathrm{Dev}}(\widetilde{\beta}) = \widetilde{\ell}_{h(x),f(x)}.$ Hence, by Lemma \ref{stab_timelike}, there exists \(m \in \mathbb{Z}\) such that 
\begin{equation}\label{eq_hol_beta}
\widetilde{\mathrm{Hol}}(\beta) = (0, 2m\pi).
\end{equation}
    On the other hand, \(T_*(\widetilde{\mathrm{Hol}}(\beta))\) induces a map from
    \[
    \mathbb{S}^1 \cong \ell_{h(x),f(x)} \to \mathbb{S}^1 \cong \ell_{h(x),f(x)},
    \]
    which induces an isomorphism \(\mathbb{Z} \to \mathbb{Z}\) on the level of the fundamental group. This is because \(T_*(\widetilde{\mathrm{Hol}}(\beta))\) is an isometry and, in particular, a homeomorphism. The isomorphism $\mathbb{Z}\to\mathbb{Z}$ is given by \(p \mapsto mp\), where \(m\) is the integer in \eqref{eq_hol_beta}. Thus, \(m = \pm 1\). However, \(m = -1\) is excluded since \(T_*(\widetilde{\mathrm{Hol}}(\beta))\) is a time-orientation-preserving isometry. This concludes the proof.
\end{proof}
Returning to the original goal of this subsection \ref{sec_fonda_ex}, by combining Lemma \ref{extension_of_dev} with Propositions \ref{holonomy_around_puncture} and \ref{holonomy_of_the_fiber}, we obtain the proof of Proposition \ref{fonda_ex}.

\subsection{Proof of Theorem \ref{principal_ads}}
In this final subsection, we prove Theorem \ref{principal_ads}.
 
\begin{proof}[Proof of Theorem~\ref{principal_ads}]
Let $\Sigma$ be a hyperbolic surface with fundamental group $\Gamma$ and let $f, h : \widetilde{\Sigma} \to \mathbb{H}^2$ be smooth maps that are equivariant with respect to representations $\rho, j : \Gamma \to \mathrm{PSL}(2,\mathbb{R})$, as in Theorem \ref{principal_ads}. As in Section \ref{sec5}, let $C \subset \widetilde{\Sigma}$ be the singular set of the branched immersion $h$, and set $C_0=C/\Gamma$. Around each point of $C$, the branched immersion $h$ is a degree-$n$ branched covering. 

Due to Theorem \ref{ads_manifold} and Lemma \ref{Mfh}, the only remaining point to verify in Theorem \ref{principal_ads} is that \( \fhadss / \Gamma_{j,\rho} \) is a branched AdS manifold.

Fix a point \( x_0 \in C \), and let \( U \subset \widetilde{\Sigma} \) be a small disc centered at \( x_0 \) on which \( h \) is a branched covering of degree \( n \). We may identify \( U \) with a small disc in \( \mathbb{H}^2 \) centered at \( i \); by applying isometries, we can further assume that \( x_0 = i \) and \( f(i) = h(i) = i \). Denote by \( \overline{f}, \overline{h} : U \to \mathbb{H}^2_* \) the restrictions of \( f \) and \( h \) respectively. Applying Proposition~\ref{fonda_ex} together with Remark \ref{local_remark}, we obtain that \( \mathbb{A}\mathrm{d}\mathbb{S}_*^{\overline{f}, \overline{h}} \) is a branched anti-de Sitter manifold that fibers over \( U \).

Since the construction of $\fhadss$ is local, the behavior of $\fhadss$ around the fiber $[\ell_{i,i}]$ is the same as that of $\mathbb{A}\mathrm{d}\mathbb{S}_*^{\overline{f},\overline{h}}$. In fact, we will show that $\mathbb{A}\mathrm{d}\mathbb{S}_*^{\overline{f},\overline{h}}$ can be identified with an open subset of the quotient $\mathcal{M}^{f,h}_* = \fhadss / \Gamma_{j,\rho}$. Indeed, one can see without difficulties that $\mathbb{A}\mathrm{d}\mathbb{S}_*^{\overline{f},\overline{h}}$ is an open set inside $\fhadss$. Considering the natural projection $p:\fhadss \to \mathcal{M}^{f,h}_*,$ we claim that, upon shrinking the neighborhood $U$, we may assume that $p:\mathbb{A}\mathrm{d}\mathbb{S}_*^{\overline{f},\overline{h}} \to p(\mathbb{A}\mathrm{d}\mathbb{S}_*^{\overline{f},\overline{h}}) \subset \mathcal{M}^{f,h}_*$ is injective. 

To prove this, we use the fact that $\Gamma$ acts properly discontinuously on $\widetilde{\Sigma}$ to shrink the open set $U$ so that $\gamma U \cap U = \emptyset$ for all $\gamma \neq e$. 
Now, by contradiction, if the restriction of $p$ to $\mathbb{A}\mathrm{d}\mathbb{S}_*^{\overline{f},\overline{h}}$ were not injective, then there would exist $A \in \mathbb{A}\mathrm{d}\mathbb{S}_*^{\overline{f},\overline{h}}$ and $\gamma \in \Gamma$ such that $\gamma \cdot A \in \mathbb{A}\mathrm{d}\mathbb{S}_*^{\overline{f},\overline{h}}$. Using the equivariant fibration $\mathcal{F}:\fhads \to \widetilde{\Sigma}$, we deduce that $\mathcal{F}(A) \in U$ and $\gamma \mathcal{F}(A) \in U$, and hence $\gamma = e$. This shows that $p:\mathbb{A}\mathrm{d}\mathbb{S}_*^{\overline{f},\overline{h}} \to p(\mathbb{A}\mathrm{d}\mathbb{S}_*^{\overline{f},\overline{h}})$ is injective. 

Therefore, $\mathbb{A}\mathrm{d}\mathbb{S}_*^{\overline{f},\overline{h}}$ is identified via $p$ with an open subset of $\mathcal{M}^{f,h}_* = \fhadss / \Gamma_{j,\rho}$. By repeating this argument around each point of the singular locus of $h$, we conclude that $\mathcal{M}^{f,h}_*$ is a branched anti-de Sitter manifold, since $\mathbb{A}\mathrm{d}\mathbb{S}_*^{\overline{f},\overline{h}}$ is one. 

Finally, we discuss the holonomy representation of $\mathcal{M}^{f,h}_*$. This is computed by studying the developing map from the universal cover $\widetilde{\fhadss}$ to $\ads$ and how it transforms under the action of $\pi_1(\mathcal{M}^{f,h}_*)$. This developing map is realized explicitly as the projection $\widetilde{\fhadss}\to \fhadss$ post-composed with the map $\mathcal{I}:\fhadss\to \ads$ defined in Section \ref{sec:local} (see also equations \eqref{D_cal} and \eqref{eq_dev_on_ads}). Since $\mathcal{F} : \mathcal{M}^{f,h}* \to \Sigma \setminus C_0$ is a circle bundle, the fundamental group $\pi_1(\mathcal{M}^{f,h}_*)$ fits into the (splitting) exact sequence $$0\to \pi_1(\mathbb{S}^1) \to \pi_1(\mathcal{M}^{f,h}_*)\to \pi_1(\Sigma\backslash C_0)\to 0.$$
By the formula for $\widetilde{\textrm{Hol}}(\beta)$ from Proposition \ref{fonda_ex}, $\widetilde{\textrm{Hol}}(\beta)$ acts trivially on $\ads$ and hence the holonomy factors through the map $\pi_1(\mathcal{M}^{f,h}_*)\to \pi_1(\Sigma\backslash C_0)$. By the formula for $\widetilde{\textrm{Hol}}(\alpha)$ from Proposition \ref{fonda_ex}, applied around each of the singular points, we see that the holonomy further factors through the map $\pi_1(\Sigma\backslash C_0)\to \pi_1(\Sigma)$ induced by inclusion. From the definition of $\mathcal{I}$ in Section \ref{sec:local} and the definition of the action of $\Gamma_{j,\rho}$, i.e., equation (\ref{eq: H action}), we see that the induced homomorphism $\pi_1(\Sigma)\to \textrm{Isom}_0(\ads)$ is $(j,\rho).$ This completes the proof.
\end{proof}

\bibliographystyle{alpha}
\bibliography{harmonic_domination}

\end{document}